\documentclass[]{article}
\usepackage[T1]{fontenc}
\usepackage[utf8]{inputenc}
\usepackage{amsmath,amssymb,amsthm}
\usepackage[top=1.3cm, bottom=1.5cm, left=3cm, right=3cm]{geometry}
\usepackage{setspace}
\usepackage{lipsum}
\usepackage{framed}
\usepackage{amsmath}
\usepackage{amsfonts}
\usepackage{amssymb}
\usepackage{amsthm}
\usepackage{mathtools}
\usepackage[utf8]{inputenc}
\usepackage{graphicx}
\usepackage{hyperref}
\usepackage{tabularx}
\usepackage{array}
\usepackage[usenames,dvipsnames]{color}
\usepackage{tikz-cd}
\usepackage{enumerate}
\usepackage{authblk}
\usepackage{multirow}
\newtheorem{theorem}{Theorem}

\newtheorem{corollary}{Corollary}
\newtheorem{lemma}{Lemma}
\newtheorem{definition}{Definition}
\newtheorem{remark}{Remark}
\newtheorem{conjecture}{Conjecture}
\newcommand\restr[2]{{
  \left.\kern-\nulldelimiterspace 
  #1 
  \vphantom{\big|} 
  \right|_{#2} 
  }}
\newcommand{\pdiv}{\mid\!\mid}

\newcommand {\ZZ}{{\mathbb{Z}}}
\newcommand {\QQ}{{\mathbb{Q}}}
\newcommand {\mcalO}{{\mathcal{O}}}

\newcommand {\mfp}{{\mathfrak{P}}}
\DeclareMathOperator{\Norm}{Norm}
\DeclareMathOperator{\im}{Im}
\DeclareMathOperator{\coker}{Coker}
\DeclareMathOperator{\Ker}{Ker}

\usepackage{color}

\title{Solving Fermat-type equations over quadratic fields}
\author{Begum Gulsah Cakti}
\date{}

\newcounter{cases}
\newcounter{subcases}[cases]

\let\oldabstract\abstract
\let\oldendabstract\endabstract
\makeatletter
\renewenvironment{abstract}
{%
               {\list{}{\addtolength{\leftmargin}{1em} 
                        \listparindent 1.5em%
                        \itemindent    \listparindent%
                        \rightmargin   \leftmargin%
                        \parsep        \z@ \@plus\p@}%
                \item\relax}%
               {\endlist}%
\oldabstract}
{\oldendabstract}
\makeatother

\begin{document}

\maketitle

    

\begin{abstract}
    This paper applies the modular approach to obtain effectively computable bounds for Fermat-type equations over number fields, while also discussing the differences and obstructions that arise when considering such equations over totally real versus totally complex number fields. We use these techniques to study the generalized Fermat equation $d^ra^p+b^p+c^p=0$ over quadratic fields $\QQ(\sqrt{d})$ of class number one. Extending the results of Freitas\&Siksek and Turcas, we show that when $d=-3,-11,-19,-43,3,5,7,11,13,19,23$, there is an effective and explicit bound, depending on the field $\QQ(\sqrt{d})$, such that the latter equation does not have certain types of special solutions. We also discuss, for $d=6,14,21$, the solutions of a variant of the above equation. Our results over imaginary quadratic fields are conjectural. Serre's modularity conjecture and an analogue of Eichler-Shimura over totally complex fields are assumed. 
\end{abstract}

\section{Introduction}

After Wiles' proof of Fermat's Last Theorem was finalized in $1995$, number theorists became eager to generalize this result via the modular approach by studying the classical Fermat equation 
\begin{equation}\label{fermat}
 x^n+y^n+z^n=0  
\end{equation}over general number fields $K$. In $2004$, Jarvis and Meekin, \cite{jarvis2004fermat}, proved that Fermat's Last Theorem holds over the field $\QQ(\sqrt{2})$. This was followed by the results of Freitas–Siksek and Michaud–Jacobs, \cite{freitas2015fermat}, \cite{michaud2021fermat}, leading to the proof that Fermat’s Last Theorem holds for real quadratic fields with small discriminant. 

After the progress on Fermat’s Last Theorem over number fields, it became natural to consider its generalizations from various angles. For example, what happens if we study Fermat-type equations of the form 
\begin{equation}\label{genfermeq}
  Ax^p+By^p+Cz^p=0  
\end{equation}where $A,B,C$ belonging to the ring of integers $\mcalO_K$ of a number field $K$? As a result of this question, the following conjecture has been proposed:

\begin{conjecture}[Asymptotic Generalized Fermat Conjecture (AGFC)]\label{AGFC}
Let $K$ be a number field, and $A$, $B$, $C$ be non-zero elements of $\mcalO_K$. Let $\Omega$ be the subgroup of roots of unity inside $\mcalO_K^{\times}$. Suppose $$A\omega_1+B\omega_2+C\omega_3\neq0\text{,  for any $\omega_1,\omega_2,\omega_3\in\Omega$}.$$Then there exists a constant $\mathcal{B}(K,A,B,C)$, depending only on $K,A,B,C$, such that for all primes $p>\mathcal{B}(K,A,B,C)$, the only solutions to the Fermat equation $$Ax^p+By^p+Cz^p=0$$with $(x,y,z)\in K^3$ are the trivial solutions. 
\end{conjecture}

The constant $\mathcal{B}(K,A,B,C)$ in the statement of Conjecture \ref{AGFC} is called \textit{effective} if it is effectively computable, and called \textit{ineffective} otherwise. In \cite{deconinck2015generalized}, Deconinck showed that AGFC holds for totally real fields whose prime ideals satisfy certain valuative conditions when the coefficients $A,B,C$ of the Equation (\ref{genfermeq}) are odd in the ring of integers of the considered field. This result was then generalized by Kara and \"{O}zman in \cite{kara2020asymptotic} to number fields satisfying usual conjectures about modularity. Note that the constants appearing in the aforementioned works of Deconinck and Kara-\"{O}zman are ineffective due to the lack of modularity results over general number fields. Therefore, despite the progress in proving AGFC, obtaining effective bounds $\mathcal{B}(K,A,B,C)$ has remained an open question.

\subsection{Statements of our results}

Our aim in this paper is to discuss the modular approach in studying Fermat-type equations over general number fields and explain the obstacles in obtaining effective bounds when proving Conjecture \ref{AGFC}. Extending the results of Freitas-Siksek \cite{freitas2015fermat} and Turcas \cite{cturcacs2020serre}, we will prove that there are effective and explicit bounds when studying generalized Fermat equations over certain real quadratic and imaginary quadratic fields. For simplicity, we will apply the modular method to the so-called $d$-Fermat equation but at the end we will also mention what happens if the setting is much more general by also highlighting the limitations of the existing techniques. 

A solution $(a, b, c) \in K^3$ to the generalized Fermat equation (\ref{genfermeq}) is called \textit{trivial} if $abc = 0$, and non-trivial otherwise. A solution $(a, b, c)$ to Equation (\ref{genfermeq}) is called \textit{integral} if $(a, b, c) \in \mcalO_K^3$. Finally, a non-trivial solution $(a, b, c)$ to Equation (\ref{genfermeq}) is called \textit{primitive} if $a$, $b$, and $c$ are pairwise coprime. Our main results are as follows:

\begin{theorem}\label{mainthm1}
Let $d=3,5,7,11,13,19,23$ and $r\in\ZZ_{\ge 1}$. Consider the generalized Fermat equation 
\begin{equation}\label{dfermat}
 d^ra^p+b^p+c^p=0  
\end{equation}over $K=\QQ(\sqrt{d})$. Assume $p\ne d$ when $d=11,13,23$, and $2\mid abc$ if $d=5,13$. Then, for any $r$, Equation (\ref{dfermat}) has no non-trivial solutions $(a,b,c)$ over $\mcalO_K$ such that $(da,b,c)$ is primitive when $p>C_K$, where $$C_K=
\left\{\begin{array}{l}
d \text {, if $d=5,7,19$ } \\
7\text {, if $d=3,11,13$ }\\
11 \text {, if $d=23$. }
\end{array}\right.$$
\end{theorem}

\begin{theorem}\label{mainthm2:d=-3}
Let $d=-3$ and $r\in\ZZ_{\ge 1}$. Assume that Conjecture \ref{serrconj} and Conjecture \ref{fakellcurve} hold for $K=\QQ(\sqrt{d})$. Then, for any $r$, Equation (\ref{dfermat}) has no non-trivial solutions $(a,b,c)$ over $\mcalO_K$ such that $(da,b,c)$ is primitive and $2\mid abc$ when $p> 13$.    
\end{theorem}

\begin{theorem}\label{mainthm3}
Let $d=-11,-19,-43$ and $r\in\ZZ_{\ge 1}$. Assume that Conjecture \ref{serrconj} holds for $K=\QQ(\sqrt{d})$. Consider Equation (\ref{dfermat}) over $K$  and assume $p\ne d$. When $2\nmid abc$, assume in addition that $p$ splits in $K$ and $p\equiv3\pmod4$. Let $B_K$ be the constant defined in Corollary \ref{bound:B_K}. Let $l_K,p_K$ be the largest prime corresponding to $K$ in Table \ref{table:cremona_output}, Table \ref{primetortable}, respectively. Then, for any $r$, Equation (\ref{dfermat}) has no non-trivial solutions $(a,b,c)$ over $\mcalO_K$ such that $(da,b,c)$ is primitive when $p>C_K=\max\{B_K,l_K,p_K\}$.
\end{theorem}

\subsection*{Acknowledgements}
I am extremely grateful to my advisor, Ekin \"{O}zman, for introducing me to the field, for her continuous support, and for her careful reading of this work. I would like to thank John Cremona for very helpful discussions on computing Bianchi modular forms and for his computational support, which allowed us to improve our bounds for two cases in Theorem \ref{mainthm3}. I would also like to thank Yasemin Kara for pointing out an error in an earlier version of this paper, and F{\i}rt{\i}na K\"{u}\c{c}\"{u}k for fruitful discussions on the cohomology of locally symmetric spaces. I would like to thank the anonymous referee for careful reading and valuable suggestions which improved the exposition of this paper. Finally, many thanks to the University of Groningen, Bernoulli Institute for their hospitality. This work was supported by T\"{U}B\.ITAK (The Scientific and Technological Research Council of T\"{u}rkiye) Research Grant 122F413 and finalized with the support of T\"{U}B\.ITAK 2214-A Grant No. 1059B142400663.

\subsection{An overview of the modular approach}

In this section, we aim to explain the main steps of the modular approach by also illustrating the differences between totally real fields and number fields that have complex embeddings. Let $K$ be a number field and let $G_K:=\operatorname{Gal}(\overline{K}/K)$ be the absolute Galois group of $K$.\\

\noindent\textbf{Frey curve: }To a hypothetical solution $(a,b,c)\in K^3$ of the Fermat-type equation, we attach an elliptic curve $E/K$. We usually require the elliptic curve $E$ to be semistable outside a small set of primes of $K$. Note that this is usually not manageable in full generality, so it sometimes leads to considering special solutions of the equation. In fact, this is the only reason we focused on special solutions in the statement of Theorem \ref{mainthm1}, Theorem \ref{mainthm2:d=-3} and Theorem \ref{mainthm3}. This way, we ensured that the elliptic curve attached to any putative solution of the 
$d$-Fermat equation is semistable away from the primes of $K$ above $2$.

\noindent\textbf{Irreducibility of $\bar{\rho}_{E, p}$: }Let $\bar{\rho}_{E, p}$ be the mod $p$ Galois representation attached to the elliptic curve $E$. We want to show that $\bar{\rho}_{E, p}$ is (absolutely) irreducible. Depending on the number of $2$-torsion points of $E$ over $K$, it is possible to prove this step by considering non-cuspidal $K$-rational points on certain modular curves at level depending on $p$. We explain and apply the existing methods to our results in Section \ref{section:irreducibility}.\\

\noindent\textbf{Modularity of the Frey curve: }Next, we require that the Frey curve $E/K$ is modular. When $K$ is totally real, it is known that all but finitely many $j$-invariants are modular. When $K$ is real quadratic all elliptic curves are modular (see \cite{freitas2015elliptic}). When $K$ has complex embeddings, it is usually necessary to assume a version of Serre's modularity conjecture (see Conjecture \ref{serrconj}).\\

\noindent\textbf{Level-lowering: }Applying level-lowering recipes or assumed conjectures, the above steps guarantee the existence of modular forms of parallel weight $2$ at level depending on the conductor of the elliptic curve $E$ but not depending on the solution $(a,b,c)$ whose representations are isomorphic to $\bar{\rho}_{E, p}$. We introduce these recipes in Section \ref{section:preliminaries}. \\

\noindent\textbf{Eliminating newforms: } To reach a contradiction, we need to eliminate each newform arising in the level-lowering step. If there are no modular forms at the given level, the result follows. However, this is usually not the case. If they are all irrational, it is possible to bound $p$ and get a contradiction. If there are rational newforms at the considered level, one can apply a version of the Eichler-Shimura theorem to get an elliptic curve $E'/K$ such that $\bar{\rho}_{E, p}\sim \bar{\rho}_{E', p}$. There are techniques to discard the previous isomorphism. We remark that when $K$ is totally real, Eichler-Shimura holds if the level of the form is squarefull. Over general number fields, it is conjectural (see Conjecture \ref{fakellcurve} ). We explain the existing techniques on this step in Section \ref{techniques}.\\

We remark that the method explained above is not unique when applying the modular approach and it is possible to relate the non-trivial solutions of a given Fermat-type equation to certain $S$-unit equations. However, since we aim to obtain effective and explicit results, we preferred the above method. The reader interested in the approach via $S$-units may refer to \cite{ozman2021s}.

\section{Preliminaries}\label{section:preliminaries}

Since we will prove Theorem \ref{mainthm1}, Theorem \ref{mainthm2:d=-3} and Theorem \ref{mainthm3} via the modular approach, we will now introduce the main results to apply the steps that we described earlier. We start with totally real case.

\subsection{Totally real fields}

When studying the solutions of Fermat-type equations over totally real fields, one of the most crucial steps of the modular method is to apply level-lowering recipes properly. For the convenience of the reader, we express below a level-lowering theorem by Freitas and Siksek:

\begin{theorem}[{\cite[Theorem 7]{freitas2015asymptotic}}]\label{levellow}
Let $K$ be a totally real field. Let $p \geq 5$ be a prime. Suppose $\mathbb{Q}\left(\zeta_p\right)^{+} \nsubseteq K$. Let $E$ be an elliptic curve over $K$ with conductor $\mathcal{N}$. Suppose $E$ is modular and $\bar{\rho}_{E, p}$ is irreducible. Denote by $\Delta_{\mathfrak{q}}$ the discriminant for a local minimal model of $E$ at a prime ideal $\mathfrak{q}$ of $K$. Let

$$
\mathcal{M}_p:=\prod_{\substack{\mathfrak{q}\pdiv \mathcal{N}\\ p| v_{\mathfrak{q}}\left(\Delta_{\mathfrak{q}}\right)}} \mathfrak{q}, \quad \mathcal{N}_p:=\frac{\mathcal{N}}{\mathcal{M}_p}
$$Suppose the following conditions are satisfied for all prime ideals $\mathfrak{q} \mid p$:
\begin{enumerate}[(i)]
\item $E$ is semistable at $\mathfrak{q}$;
\item\label{item:ii-levellowering} $p \mid v_\mathfrak{q}\left(\Delta_\mathfrak{q}\right)$;
\item the ramification index satisfies $e(\mathfrak{q} / p)<p-1$.
\end{enumerate}

Then, $\bar{\rho}_{E, p} \sim \bar{\rho}_{\mathfrak{f}, \omega}$ where $\mathfrak{f}$ is a Hilbert eigenform of parallel weight 2 that is new at level $\mathcal{N}_p$ and $\omega$ is a prime ideal of $\mathbb{Q}_f$ that lies above $p$.
\end{theorem}

In our case, the elliptic curve in consideration will be modular due to a result of Freitas, Le Hung and Siksek:

\begin{theorem}[{\cite[Theorem 1]{freitas2015elliptic}}]\label{box}
Let $E$ be an elliptic curve over a real quadratic field $K$. Then $E$ is modular.
\end{theorem}

\subsection{Totally complex fields}

Next, we present two crucial conjectures when studying Diophantine equations over general number fields. We start with a version of Serre's modularity conjecture, which we need to assume due to lack of modularity over number fields with complex embeddings.

\begin{conjecture}[Serre's modularity conjecture-{\cite[Conjecture 3.1]{csengun2018asymptotic}}]\label{serrconj}
Let $\bar{\rho}: G_K \rightarrow \mathrm{GL}_2\left(\overline{\mathbb{F}}_p\right)$ be an irreducible, continuous representation with Serre conductor $\mathcal{N}$ (prime-to-$p$ part of its Artin conductor) and trivial character (prime-to-p part of $\operatorname{det}(\bar{\rho}))$. Assume that $p$ is unramified in $K$ and that $\left.\bar{\rho}\right|_{G_{K_{\mathfrak{p}}}}$ arises from a finite-flat group scheme over $\mathcal{O}_{K_{\mathfrak{p}}}$ for every prime $\mathfrak{p} \mid p$. Then there is a mod $p$ eigenform $\Psi: \mathbb{T}_{\mathbb{F}_p}\left(Y_0(\mathcal{N})\right) \rightarrow \overline{\mathbb{F}}_p$ such that for all prime ideals $\mathfrak{q} \subseteq \mathcal{O}_K$, coprime to $p \mathcal{N}$
$$
\operatorname{Tr}\left(\bar{\rho}\left(\operatorname{Frob}_{\mathfrak{q}}\right)\right)=\Psi\left(T_{\mathfrak{q}}\right) .
$$
\end{conjecture}

We remark that since our aim in Theorem \ref{mainthm2:d=-3} and Theorem \ref{mainthm3} is to study $d$-Fermat equation over imaginary quadratic fields, one may wonder why the below result of Caraiani-Newton is not enough.

\begin{theorem}[{\cite[Theorem 1.1]{caraiani2023modularity}}]\label{modularity-imaginaryquadratic}
Let $F$ be an imaginary quadratic field such that the Mordell-Weil group $X_0(15)(F)$ is finite. Then every elliptic curve $E / F$ is modular.
\end{theorem}

Although Theorem \ref{modularity-imaginaryquadratic} provides the modularity of elliptic curves over certain imaginary quadratic fields, which is a major step forward, modularity alone does not typically suffice for applying modular method arguments since one still needs a level-lowering result in the general number field setting. This is the reason why we assume Conjecture \ref{serrconj} in Theorem \ref{mainthm2:d=-3} and Theorem \ref{mainthm3}.

The following conjecture arises from the lack of the Eichler-Shimura theorem over general number fields, and also needs to be assumed in Theorem \ref{mainthm2:d=-3}.

\begin{conjecture}[{\cite[Conjecture 4.1]{csengun2018asymptotic}}]\label{fakellcurve}
Let $\mathfrak{f}$ be a (weight 2) complex eigenform over $K$ of level $\mathfrak{N}$ that is non-trivial and new. If $K$ has some real place, then there exists an elliptic curve $E_{\mathfrak{f}} / K$, of conductor $\mathfrak{N}$, such that

\begin{equation}\label{ellcurveconj}
 \# E_{\mathfrak{f}}\left(\mathbb{Z}_K / \mathfrak{q}\right)=1+\mathbf{N q}-\mathfrak{f}\left(T_{\mathfrak{q}}\right) \quad \text { for all } \mathfrak{q} \nmid \mathfrak{N} .   
\end{equation}

If $K$ is totally complex, then there exists either an elliptic curve $E_{\mathfrak{f}}$ of conductor $\mathfrak{N}$ satisfying (\ref{ellcurveconj}) or a fake elliptic curve $A_{\mathfrak{f}} / K$, of conductor $\mathfrak{N}^2$, such that

\begin{equation}\label{fakeellcurveconj}
\# A_{\mathfrak{f}}\left(\mathbb{Z}_K / \mathfrak{q}\right)=\left(1+\mathbf{N q}-\mathfrak{f}\left(T_{\mathfrak{q}}\right)\right)^2 \quad \text { for all } \mathfrak{q} \nmid \mathfrak{N} \text {. }
\end{equation}
\end{conjecture}

\section{Frey curve attached to \texorpdfstring{$d$}{d}--Fermat equation}

\subsection{Notation}

Let $d=-43, -19, -11,-3,3,5,7,11,13,19,23$. Let $K=\QQ(\sqrt{d})$ and denote the ring of integers of $K$ by $\mcalO_K$. 

Throughout the paper, we refer to the equation
\begin{equation*}\label{GFE}
    d^ra^p+b^p+c^p=0
\end{equation*} as \textit{$d$-Fermat equation of exponent $p$}. Recall that a solution $(a, b, c) \in K^3$ of Equation (\ref{GFE}) is called trivial if $abc = 0$ and non-trivial otherwise. A solution $(a, b, c)$ of Equation (\ref{GFE}) is called integral if $(a, b, c) \in \mcalO_K^3$. Finally, when $(a, b, c)$ is a non-trivial solution of Equation (\ref{GFE}), the triple $(da,b,c)$ called primitive if $da$, $b$, and $c$ are pairwise coprime. Note that since we are considering the number fields that have class number one, it is possible to scale the solution $(a, b, c)$ to remain integral and primitive.

Note that $2$ is ramified in $K$ when $d=3,7,11,19,23$ and is inert in $K$ otherwise, i.e, irrespective of the choice of $d$, there is a unique prime above $2$ in $K$. We will denote the unique prime above $2$ by $\mfp$.

When $d=3,7,11,19,23$, we have $2\mcalO_K=\mfp^2$. It follows that $\mcalO_K/\mfp \cong \mathbb{F}_2$. Since $\gcd(da,b,c)=1$ and $d$ is odd, the prime $\mfp$ divides precisely one of $a$, $b$ and $c$. On the other hand, the residue field of $2$ is not $\mathbb{F}_2$ when $d\ne 3,7,11,19,23$ since $2$ is inert in $K$ in this case. Recall that we are considering the solutions $(a,b,c)\in\mcalO_K^3$ such that $2\mid abc$ and $(da,b,c)$ is primitive. Therefore, we again obtain that $\mfp=2\mcalO_K$ divides precisely one of $a$, $b$ and $c$ when $d\ne 3,7,11,19,23$. We remark that this is the only reason why we had to assume $2\mid abc$ in the latter cases. In all cases, without loss of generality, we will always assume $\mfp\mid b$.

\subsection{Semistability of the Frey curve}
Let $(a,b,c)\in\mcalO_K^3$ be a non-trivial solution of the $d$-Fermat equation (\ref{dfermat}) such that $(da,b,c)$ is primitive or $(da,b,c)$ is primitive and $2\mid abc$, depending on which case we are considering in Theorem \ref{mainthm1}, Theorem \ref{mainthm2:d=-3} and Theorem \ref{mainthm3}. The Frey curve attached to this solution is given by

\begin{equation}\label{frey}
   \text{$E:=E_{a,b,c}$:      } y^2=x(x-d^r a^p)(x+b^p)
\end{equation}.

and it has the following invariants:

\begin{align*}
   \Delta_E &=2^{4}d^{2r}(abc)^{2p}\\
c_4 &= 2^4(b^{2p}-d^ra^pc^p)\\
c_6 &= -2^5(d^ra^p-b^p)(b^p-c^p)(c^p-d^ra^p)\\
j_E&=\frac{c^3_4}{\Delta_E}=2^8\frac{(b^{2p}-d^ra^pc^p)^3}{d^{2r}(abc)^{2p}}
\end{align*}

We start by discussing the semistability of the elliptic curve $E$, validating the choice of considering special solutions in Theorem \ref{mainthm1}, Theorem \ref{mainthm2:d=-3} and Theorem \ref{mainthm3}.

\begin{lemma}\label{semsq}
The Frey curve $E$ is semistable at primes above $d$.   
\end{lemma}

\begin{proof}
  Let $\mathfrak{p}$ be a prime above $d$. Clearly, $\mathfrak{p}\mid \Delta_E$. Since $(da,b,c)$ is a primitive solution by assumption, we have $$v_{\mathfrak{p}}(c_4)=\min(2pv_{\mathfrak{p}}(b), p(v_{\mathfrak{p}}(a)+v_{\mathfrak{p}}(c))+r)=0,$$so that $E$ has multiplicative reduction at $\mathfrak{p}$ hence is semistable at $\mathfrak{p}$.
\end{proof}

\begin{lemma}\label{Esem}
 E is semistable away from $\mfp$.
\end{lemma}

\begin{proof}
Let $\mathfrak{p}\ne \mfp$ be a prime of $K$. Recall that $\Delta_E = 2^{4}d^{2r}(abc)^{2p}$ and $c_4 = 2^4(b^{2p}-d^ra^pc^p)$. Assume $\mathfrak{p}\mid \Delta_E$. Then either $\mathfrak{p}\mid d$ or $\mathfrak{p}\mid abc$. If $\mathfrak{p}\mid d$, then by Lemma \ref{semsq}, we know that $E$ is semistable at $\mathfrak{p}$. If $\mathfrak{p}\mid abc$, then observe that $\mathfrak{p}$ divides at most one of $a,b,c$. Indeed, if $\mathfrak{p} $ divides two of them, then the relation $d^ra^p+b^p+c^p=0$ yields $\mathfrak{p} =\mcalO_K$, which is not possible. Therefore, $\mathfrak{p}\nmid (b^{2p}-d^ra^pc^p)$, i.e. $v_{\mathfrak{p}}(c_4)=0$ and $E$ is semistable at $\mathfrak{p}$. 
\end{proof}

The following lemma is needed when showing the irreducibility of the mod $p$ representation attached to curve $E$ and at the stage of newform elimination. First, we deal with the case where either residual degree of $\mfp$ in $K$ is $1$ or $2$ is inert in $K$ with $2\mid abc$. The remaining case will be treated separately at the end of the section.

\begin{lemma}\label{potmultb}
Let $K$ be a quadratic field and let $\mfp\mid 2$. Assume that either the residual degree of the prime $\mfp$ in $K$ is one or $2$ is inert in $K$ with $2\mid abc$. The Frey curve $E/K$ has potentially multiplicative reduction at the prime $\mfp$.
\end{lemma}

\begin{proof}
    It suffices to show that $v_{\mfp}(j_E)<0$ by \cite[Proposition VII.5.5]{silverman2009arithmetic}. Recall that, in all cases, we assumed $\mfp\mid b$. Then, $v_{\mfp}(b)\ge1$. Observe that 
    \begin{eqnarray*}
    v_{\mfp}(c_4)&=&v_{\mfp}(2^4(b^{2p}-d^ra^pc^p))\\
    &=&4v_{\mfp}(2)+v_{\mfp}(b^{2p}-d^ra^pc^p)\\
    &=&4v_{\mfp}(2)+\min(2pv_{\mfp}(b), p(v_{\mfp}(a)+v_{\mfp}(c)))\\
    &=&4v_{\mfp}(2)
\end{eqnarray*}and that
    \begin{eqnarray*}
    v_{\mfp}(\Delta_E)&=& v_{\mfp}(2^4d^{2r}(abc)^{2p})\\
    &=&4v_{\mfp}(2)+2rv_{\mfp}(d)+2p(v_{\mfp}(a)+v_{\mfp}(b)+v_{\mfp}(c))\\
    &=&4v_{\mfp}(2)+2pv_{\mfp}(b)
    \end{eqnarray*}Therefore,$$v_{\mfp}(j_E)=3v_{\mfp}(c_4)-v_{\mfp}(\Delta_E)=8v_{\mfp}(2)-2pv_{\mfp}(b)<0$$since in every case, we assumed $p> 5$ when $2$ is inert and $p>7$ when $2$ is ramified in $K$.
\end{proof}

\begin{lemma}\label{pdivdel}
Let $p$ be a rational prime. Then, $p\mid v_{\mathfrak{p}}(\Delta_E)$ for any prime ideal $\mathfrak{p}$ of $K$ such that $\mathfrak{p}\nmid 2d$.    
\end{lemma}

\begin{proof}
    Let $\mathfrak{p}\nmid 2d$ be a prime of $K$. Observe that 
    \begin{align*}
      v_{\mathfrak{p}}(\Delta_E)&=4v_{\mathfrak{p}}(2)+2rv_{\mathfrak{p}}(d)+2pv_{\mathfrak{p}}(abc)\\
      &=2pv_{\mathfrak{p}}(abc).
    \end{align*}Therefore, $p\mid v_{\mathfrak{p}}(\Delta_E)$.
\end{proof}

From Lemma \ref{potmultb}, observe that the elliptic curve $E$ has potentially multiplicative reduction at $\mfp$ when we are in one of the cases in Theorem \ref{mainthm1} or one of the cases in Theorem \ref{mainthm3} where the condition $2\mid abc$ is assumed. Now, we will determine the exponent of the prime $\mfp$ in the conductor $N_E$ of the elliptic curve $E$ for those cases. We will closely follow the arguments in \cite[Lemma 4.3 \& Lemma 4.4]{freitas2015fermat}. Although the aforementioned work is based on the classical Fermat equation over real quadratic fields, the computations done in these lemmas are purely local and applicable to Frey curves corresponding to the solutions of generalized Fermat equations and to those of Fermat-type equations over imaginary quadratic fields. For the sake of the reader, we first state \cite[Lemma 4.3]{freitas2015fermat}:

\begin{lemma}[{\cite[Lemma 4.3]{freitas2015fermat}}]\label{samir4.3}
    Let $K_{\mathfrak{P}}$ be a local field, and $E$ be an elliptic curve over $K_{\mathfrak{P}}$ with potentially multiplicative reduction. Let $c_4, c_6$ be the usual c-invariants of E. Let $L=K_{\mathfrak{P}}\left(\sqrt{-c_6 / c_4}\right)$ and let $\Delta\left(L / K_{\mathfrak{P}}\right)$ be the discriminant of this local extension. Then the conductor of $E / K_{\mathfrak{P}}$ is
$$
f\left(E / K_{\mathfrak{P}}\right)= \begin{cases}1 & \text { if } v_{\mathfrak{P}}\left(\Delta\left(L / K_{\mathfrak{P}}\right)\right)=0 \\ 2 v_{\mathfrak{P}}\left(\Delta\left(L / K_{\mathfrak{P}}\right)\right) & \text { otherwise. }\end{cases}
$$
\end{lemma}

Recall from Lemma \ref{potmultb} that our Frey curve $E$ has potentially multiplicative reduction at the prime $\mfp$. We will state and prove a version of \cite[Lemma 4.4]{freitas2015fermat} that is applied to our case.

\begin{lemma}\label{samir4.4}
    Let $K$ be a quadratic field and let $\mfp\mid 2$. Assume that either the residual degree of the prime $\mfp$ in $K$ is $1$ or $2$ is inert in $K$ with $2\mid abc$. Let $(a,b,c)$ be a non-trivial solution to the $d$-Fermat equation $(\ref{dfermat})$. Define the ideal $\mathfrak{b}=\prod_{\substack{\mfp \mid 2  }}\mfp^{2v_{\mfp}(2)+1}$ and let $$
\Phi: \mathcal{O}_K^{\times} \rightarrow\left(\mathcal{O}_K / \mathfrak{b}\right)^{\times} /\left(\left(\mathcal{O}_K / \mathfrak{b}\right)^{\times}\right)^2
$$
be the natural map. Choose elements $\lambda_1, \ldots, \lambda_k \in \mathcal{O}_K \backslash \mathfrak{b}$ that represent the elements of the cokernel of $\Phi$. For $1 \leq i \leq k$, let $\Delta_{\mfp}^{(i)}$ be the discriminant of the local extension $K_{\mfp}\left(\sqrt{\lambda_i}\right) / K_{\mfp}$, and let
$$
\epsilon_{\mfp}^{(i)}= \begin{cases}1 & \text { if } v_{\mfp}\left(\Delta_{\mfp}^{(i)}\right)=0 \\ 2 v_{\mfp}\left(\Delta_{\mfp}^{(i)}\right) & \text { otherwise. }\end{cases}
$$

Then we may scale $(a, b, c)$ by an element of $\mathcal{O}_K^{\times}$ so that for some $i$, and every prime $\mfp$ above $2$, we have $v_{\mfp}(N_E)=\epsilon_{\mfp}^{(i)}$.
\end{lemma}

\begin{proof}
    Write $\mcalO:=\mcalO_K$. The above computations on the valuation of the invariants $c_4$, and $c_6$ show that $v_{\mfp}(c_4)=4v_{\mfp}(2)$, and $v_{\mfp}(c_6)=6v_{\mfp}(2)$. Define $\gamma:=\frac{-c_6}{4c_4}$ and observe that $v_{\mfp}(\gamma)=0$. Therefore, $\gamma\in \mcalO_{\mfp}^{\times}$ and $K_{\mfp}(\sqrt{\gamma})=K_{\mfp}(\sqrt{-c_6/c_4})$. Since $E$ has potentially multiplicative reduction at $\mfp$, Lemma \ref{samir4.3} is applicable, from which it follows that $v_{\mfp}(N_E)$ can be expressed in terms of the discriminant of the extension $K_{\mfp}(\sqrt{\gamma})/K_{\mfp}$. From Chinese remainder theorem and Hensel's lemma, we have the isomorphism $$\left(\mathcal{O}_K / \mathfrak{b}\right)^{\times} /\left(\left(\mathcal{O}_K / \mathfrak{b}\right)^{\times}\right)^2 \cong \prod_{\substack{\mathfrak{\mfp} \in S  }} \mcalO_{\mfp}^{\times} / (\mcalO_{\mfp}^{\times})^2.$$Observe that if we scale the solution $(a,b,c)$ by a unit $\alpha \in \mcalO^{\times}$, then the Frey curve attached to the corresponding solution has invariants $c_4^{\prime}=\alpha^{2p}c_4$ and $c_6^{\prime}=\alpha^{p}c_6$. This shows that scaling the solution $(a,b,c)$ by $\alpha$ results in scaling $\gamma$ by $\alpha^p$. Since $p$ is odd, it follows
from the definition of $\Phi$ and the above isomorphism that we can scale $(a, b, c)$ by some $\alpha\in \mcalO^{\times}$ so that $\gamma / \lambda_{\mfp}^i$ is a square in $\mcalO_{\mfp}$ for some $i\in \{1,2,\ldots,k\}$ for each prime $\mfp$ above $2$. Thus, $K_{\mfp}(\sqrt{\gamma})=K_{\mfp}\left(\sqrt{\lambda_{\mfp}^i}\right)$ for all prime $\mfp$ above $2$ and the result follows from Lemma \ref{samir4.3}.
\end{proof}

As indicated in \cite{freitas2015fermat}, if $K$ is a real quadratic field and $u$ is the fundamental unit of $K$, it is possible to replace each representative $\lambda\in\mcalO_K\setminus\mathfrak{b}$ by one of $\pm\lambda, \pm u^k\lambda$ to minimize the norm and control the lowered level. When $K$ is an imaginary quadratic field, it is convenient to consider the maximum of the valuation $v_{\mfp}\left(\Delta_{\mfp}^{(i)}\right)$ as the exponent of $\mfp$ in the conductor of the elliptic curve in consideration since the lower possibilities need to be taken care of anyways due to Serre's modularity conjecture. 

For each of the fields considered, we list the ideal $\mathfrak{b}$, the representatives of the cokernel of $\Phi$, and the exponent of $\mfp$ in the conductor $N_E$ in Table \ref{representativetable}. For example, when $d=3$, we observe that $\operatorname{Coker}(\Phi)\cong \ZZ/2\ZZ$ and $\lambda_1=1,\lambda_2=-1+2\sqrt{3}$ can be chosen as the representatives of the cokernel of $\Phi$. For $i=1,2$, let $n_i=\operatorname{ord}_{\mathfrak{P}}\left(\Delta\left(L_i / K\right)\right)$ where $L_i=K\left(\sqrt{\lambda_i}\right)$ and $\Delta\left(L_i / K\right)$ is the relative discriminant ideal for the extension $L_i / K$. Using the relevant code in \cite{Khawaja_Fermat} and Magma, we see that $n_1=0$ and $n_2=2$. By Lemma \ref{samir4.4}, we can scale the solution $(a,b,c)$ by a unit $\varepsilon \in \mathcal{O}_K^{\times}$ so that $v_{\mathfrak{P}}\left(N_E\right)=1$ or $4$. When $d=-3$, we observe that $\operatorname{Coker}(\Phi)\cong \left(\ZZ/2\ZZ\right)^2$ and $\lambda_1=1, \lambda_2=-2+3\sqrt{-3}, \lambda_3=1+4 \sqrt{-3}, \lambda_4=2-\sqrt{-3}$ can be chosen as representatives the cokernel of $\Phi$. As discussed in the previous paragraph, since $\max _{1 \leq i \leq 4} v_{\mfp}\left(\Delta\left(K_{\mfp}\left(\sqrt{\lambda_i}\right) / K_{\mfp}\right)\right)=2$, using Lemma \ref{samir4.4}, it is possible to scale to scale the solution $(a,b,c)$ by a unit $\varepsilon \in \mathcal{O}_K^{\times}$ so that $v_{\mathfrak{P}}\left(N_E\right)=4$.

\begin{table}
$$\begin{array}{|c|c|c|c|}
\hline d & \text{Ideal $\mathfrak{b}$}  &\text { Representatives } \lambda_i \in \mathcal{O}_K \text { of Coker }(\Phi) & v_{\mfp}(N_E) \\
\hline 3 & \mfp^5  & 1,-1+2\sqrt{3} & 1,4\\
\hline 5 & \mfp^3 & 1,3+\sqrt{5} & 1,4\\
\hline 7 & \mfp^5  & 1,- 21 -8\sqrt{7},-1+2\sqrt{7},37+14\sqrt{7}& 1,4\\
\hline 11 & \mfp^5  & 1,-1+2\sqrt{11} & 1,4\\
\hline 13 & \mfp^3 & 1, 1+4\sqrt{13}& 1\\
\hline 19 & \mfp^5 & 1, -1+2\sqrt{19}& 1,4\\
\hline 23 & \mfp^5 & 1, -115-24\sqrt{23},-1+2\sqrt{23},163+34\sqrt{23}& 1,4\\
\hline -3 & \mfp^3 & 1, -2+3\sqrt{-3}, 1+4 \sqrt{-3}, 2-\sqrt{-3} & 1,4\\
\hline -11 & \mfp^3 & 1, -1+\sqrt{-11},-3+4 \sqrt{-11},  -1-3 \sqrt{-11} & 1,4\\
\hline -19 & \mfp^3 & 1, -3+\sqrt{-19}, -3+4\sqrt{-19}, -3-3\sqrt{-19}& 1,4\\
\hline -43 & \mfp^3 & 1, 1+4\sqrt{-43}, 4+3\sqrt{-43}, -\sqrt{-43} & 1,4\\
\hline
\end{array}$$
\caption{Numerical values associated to Lemma \ref{samir4.4} }
\label{representativetable}
\end{table}

Next, let us treat the case where $2$ is inert in $K$ and $2\nmid abc$. Recall that when $2\mid abc$, we made use of \cite[Lemma 4.4]{freitas2015fermat} to determine the conductor exponent of the curve $E$ at $\mfp$, whose proof depends on the fact that $E$ has potentially multiplicative reduction at $\mfp$. The following lemma shows why \cite[Lemma 4.4]{freitas2015fermat} is not applicable in this setting.

\begin{lemma}\label{lemma: Epotgoodred}
    Let $d=-3,-11,-19,-43$ and assume that $2\nmid abc$. The elliptic curve $E$ has potentially good reduction at $\mfp$.
\end{lemma}

\begin{proof}
   Let $d=-3,-11,-19,-43$ and assume that $2\nmid abc$. Note that 
  \begin{eqnarray*}
    v_{\mfp}(c_4)&=&v_{\mfp}(2^4(b^{2p}-d^ra^pc^p))\\
    &=&4v_{\mfp}(2)+v_{\mfp}(b^{2p}-d^ra^pc^p)\\
    &=&4+v_\mfp(b^{2p}-d^ra^pc^p)\\
    &\ge& 5
\end{eqnarray*}and that
    \begin{eqnarray*}
    v_{\mfp}(\Delta_E)&=& v_{\mfp}(2^4d^{2r}(abc)^{2p})\\
    &=&4v_{\mfp}(2)+2rv_{\mfp}(d)+2p(v_{\mfp}(a)+v_{\mfp}(b)+v_{\mfp}(c))\\
    &=&4
    \end{eqnarray*}Therefore,$$v_{\mfp}(j_E)\ge 11>0$$and $E$ has potentially good reduction at $\mfp$.
\end{proof}

From the proof of the previous lemma, we see that $v_{\mfp}(\Delta_E)=4$ and $v_{\mfp}(c_4)\ge 5$ when $d=-3,-11,-19,-43$ and $2\nmid abc$. By \cite[Tableau IV]{papadopoulos1993neron}, we observe that $v_\mfp(N_E)=2,3$ or $4$. We explain below that it is possible to permute $(a,b,c)$ so that $v_\mfp(N_E)=4$.

\begin{lemma}
    Let $d=-3,-11,-19,-43$ and assume that $2\nmid abc$. Then, after permuting $(a,b,c)$, we have $v_\mfp(N_E)=4$.
\end{lemma}

\begin{proof}
In this case, the valuation ring at $\mfp$ is $\mathcal{O}_{\mathfrak{P}}=\mathbb{Z}_2[\alpha]$ where $\alpha$ is a root of the polynomial $x^2+x+1=0$. The residue field of $\mathfrak{P}$ is $\mathbb{F}_2[\tilde{\alpha}] \cong \mathbb{F}_4$. Write $A=d^ra^p, B=b^p, C=c^p$. Now $\mathfrak{P} \nmid A B C$ and $A+B+C=0$. Then $A, B, C$ are congruent modulo $\mathfrak{P}$ to $1$ , $\alpha, \alpha^2$ in some order. By rearranging, we may suppose that $C \equiv \alpha^2(\bmod \mathfrak{P})$. Note that $v_{\mathfrak{P}}(\Delta_E)=4$ and $v_{\mathfrak{P}}\left(c_4\right)\ge 5$. In particular, $v_{\mathfrak{P}}(j)\ge 11$, i.e $E$ has potentially good reduction at $\mathfrak{P}$. Furthermore, the model for $E$ is minimal at $\mathfrak{P}$ and has additive reduction. We will proceed with Tate's algorithm. Let $\tilde{E}$ denote the reduction of $E$ modulo $\mfp$. Observe that the point $(\tilde{C}, \tilde{1})$ is singular on $\tilde{E}$. Indeed, write $f(x)=x(x-\tilde{A})(x+\tilde{B})$. Then

\begin{eqnarray*}
    \frac{\partial }{\partial x}(y^2-f(x))&=&3x^2+2(\tilde{B}-\tilde{A})x-\tilde{A}\tilde{B}\\
    &=&x^2-\tilde{A}\tilde{B}\text{ since  $\operatorname{char}(\mathbb{F}_4)=2$.}\\
\end{eqnarray*}This shows that

$$\left.\frac{\partial }{\partial x}(y^2-f(x))\right|_{(\tilde{C}, \tilde{1})}=\tilde{C}^2-\tilde{A}\tilde{B}=\tilde{\alpha}^4-\tilde{\alpha}=\tilde{\alpha}-\tilde{\alpha}=0.$$We also have $$\frac{\partial }{\partial y}(y^2-f(x))=2y=0\text{ since  $\operatorname{char}(\mathbb{F}_4)=2$.}$$Apply the transformation $X\mapsto X+C$, $Y\mapsto Y+1$ to $E$ and denote the usual $a$-invariants of the resulting model by $a_1, \ldots, a_6$. Then $a_6=C^3+(B-A) C^2-A B C-1$. By Step $3$ of Tate's algorithm we know that if $\mathfrak{P}^2 \nmid a_6$ then $v_{\mathfrak{P}}(\mathcal{N})=v_{\mathfrak{P}}(\Delta)=4$. Suppose $\mathfrak{P}^2 \mid a_6$. Now swapping $A, B$, replaces $a_6$ by $a_6^{\prime}=C^3+(A-B) C^2-A B C-1.$ Observe that $v_{\mathfrak{P}}\left(a_6^{\prime}-a_6\right)=v_{\mathfrak{P}}\left(2(A-B) C^2\right)=1$. Hence $\mathfrak{P}^2 \nmid a_6^{\prime}$. Thus we may permute $A,B,C$ so that $v_\mfp(N_E)=4$.
    
\end{proof}

We may now explicitly express the conductor of the Frey curve $E$:

\begin{lemma}\label{condE}
    The conductor of the elliptic curve $E/K$ is given by $$N_E=\mfp^l\prod_{\substack{\mathfrak{D} \mid d }} \mathfrak{D} \prod_{\substack{\mathfrak{p} \mid a b c \\ \mathfrak{p} \nmid 2d}} \mathfrak{p},$$where $l=
\left\{\begin{array}{l}
1,4 \text {, if $d=3,5,7,11,19,23$ } \\
1 \text {, if $d=13$ }\\
4 \text {, if $d=-3,-11,-19,-43$. }
\end{array}\right.$
\end{lemma}

\begin{proof}
    Recall that $\Delta_E = 2^4d^{2r}(abc)^{2p}$. We know that the primes appearing in the conductor of $E$ are the ones dividing $\Delta_E$, hence dividing $2d$ or $abc$. For the prime $\mfp$, we explained above that the exponent $l$ is precisely how it is claimed. By the proof of Lemma \ref{semsq}, we know that $E$ has multiplicative reduction at the primes above $d$. Therefore, $v_{\mathfrak{D}}(N_E)=1$ for all $\mathfrak{D}\mid d$. Now, let $\mathfrak{p}\mid abc$ such that $\mathfrak{p}\nmid 2d$. Then, $\mathfrak{p}$ divides at most one of $a,b,c$. Indeed, if $\mathfrak{p} $ divides two of them, then the relation $d^ra^p+b^p+c^p=0$ yields $\mathfrak{p} =\mcalO_K$, which is not possible. Therefore, $\mathfrak{p}\nmid (b^{2p}-d^ra^pc^p)$, i.e. $v_{\mathfrak{p}}(c_4)=0$. Thus, $E$ has multiplicative reduction at $\mathfrak{p}$ and the result follows.
\end{proof}

\section{Galois representations attached to elliptic curves}\label{galoisrepr}

Let $G_K:=\operatorname{Gal}(\overline{K}/K)$ be the absolute Galois group of the number field $K$. Let
$$
\bar{\rho}_{E, p}: G_K \longrightarrow \operatorname{Aut}(E[p]) \cong \mathrm{GL}_2\left(\mathbb{F}_p\right)
$$
be the $\bmod\text{ } p$ Galois representation of $E$.

Let us focus on the cases where $d=-3,-11,-19,-43$. Let $S=\{\mathfrak{p}:\mathfrak{p}\mid 2d\}$ and let $p\ge5$ be a rational prime that is not divisible by any $\mathfrak{p}\in S$. Regarding Theorem \ref{mainthm3}, by \cite[Corollary 3.5]{kara2020asymptotic}, we observe that the mod $p$ representation $\bar{\rho}_{E, p}$ attached to the elliptic curve $E/K$ with $K=\QQ(\sqrt{d})$ satisfy the following properties: 
\begin{enumerate}
    \item The determinant of $\bar{\rho}_{E, p}$ is the mod $p$ cyclotomic character, hence it is odd.
    \item Serre conductor of $\bar{\rho}_{E, p}, \mathcal{N}$, is supported on $S$ and belongs to a finite set that depends only on the field $K$.
    \item $\bar{\rho}_{E, p}$ is finite flat at every $\mathfrak{q}$ over $p$.
\end{enumerate}

By the definition of Serre conductor $\mathcal{N}$, we need to find out the Artin conductor of $\bar{\rho}_{E, p}$ to determine $\mathcal{N}$. We explicitly compute the Artin conductor $\mathcal{N}_p$ of the representation $\bar{\rho}_{E, p}$ by using the recipe from \cite{kraus1997determination}: We have 

$$\mathcal{N}_p=\prod_{\substack{\mathfrak{p} \mid N_E \\ \mathfrak{p} \nmid p}} \mathfrak{p}^{v_\mathfrak{p}(N_E)-f_\mathfrak{p}}$$where the quantity $f_\mathfrak{p}$ is calculated as follows:

\begin{enumerate}[(i)]
\item If $E$ has good or additive reduction at $\mathfrak{p}$, then $f_\mathfrak{p}=0$.
\item If $E$ has multiplicative reduction at $\mathfrak{p}$, then

$$
f_\mathfrak{p}= \begin{cases}0 & \text { if } p \text { does not divide } v_\mathfrak{p}(\Delta_E) \\ 1 & \text { if } p \text { divides } v_\mathfrak{p}(\Delta_E)\end{cases}
$$
\end{enumerate}

Recall that the elliptic curve $E$ has additive reduction at the prime $\mfp$ with $v_{\mfp}(N_E)=4$. It follows that $f_\mfp=0$ and we definitely have $\mfp^4\mid \mathcal{N}_p$ for each rational prime $p$. Next, we determine the exponent of the primes $\mathfrak{D}\mid d$ in the Serre conductor $\mathcal{N}_p$. Note that $d$ ramifies in $\QQ(\sqrt{d})$, there is always a unique prime above $d$. Let us call that prime $\mathfrak{D}$ and recall that $E$ has multiplicative reduction at $\mathfrak{D}$ by Lemma \ref{semsq}. Therefore, $v_{\mathfrak{D}}(N_E)=1$ and $f_\mathfrak{D}=0,1$. Note that \begin{align*}
      v_{\mathfrak{D}}(\Delta_E)&=4v_{\mathfrak{D}}(2)+2rv_{\mathfrak{D}}(d)+2pv_{\mathfrak{D}}(abc)\\
      &=2r+2pv_{\mathfrak{D}}(abc)
    \end{align*}Notice that if $p\mid r$, then $p\mid v_{\mathfrak{D}}(\Delta_E)$ and  $p\nmid v_{\mathfrak{D}}(\Delta_E)$ otherwise. Hence, we deduce that $$
\mathcal{N}_p= \begin{cases}\mfp^4 & \text { if } p \text {  divides } r \\ \mfp^4\mathfrak{D} & \text { if } \gcd(p,r)>1 \end{cases}
$$Since the case $\mathcal{N}_p=\mfp^4$ was already treated by Turcas in \cite{cturcacs2020serre}, from now on we consider the case where $p\nmid r$. Therefore, $\mathcal{N}_p=\mfp^4\mathfrak{D}$ for all rational prime $p\ge 5$.

When $K$ is totally real, the first property above also holds for $\bar{\rho}_{E, p}$.

\subsection{Irreducibility}\label{section:irreducibility}

In this section, we prove that the representation $\bar{\rho}_{E, p}$ is irreducible for $p\ge 7$ by treating the cases where 
$p=7$, $p=11$, $p=13$, $p=17$, and $p\ge 19$ separately. Our strategy is to assume that $\bar{\rho}_{E, p}$ is reducible and search for a contradiction. For the cases where $p=7,11,13,17$, we will make use of the following reasoning: Since the elliptic curve $E$ has full $2$-torsion over $K$, the Galois representation $\bar{\rho}_{E, p}$ gives rise to a non-cuspidal $K$-rational point on the modular curves $X_0(p)$, $X_0(2p)$ and $X_0(4p)$ when $\bar{\rho}_{E, p}$ is reducible. Therefore, it suffices to show that  $X_0(p)(K)$, $X_0(2p)(K)$ or $X_0(4p)(K)$ has no points that could arise from the Frey curve $E$. 

Note that when the representation $\bar{\rho}_{E, p}$ is reducible, it is of the following form:

$$
\bar{\rho}_{E, p} \sim\left(\begin{array}{cc}
\theta & * \\
0 & \theta^{\prime}
\end{array}\right) \text { with } \theta, \theta^{\prime}: G_K \rightarrow \mathbb{F}_p^{\times}  \text { satisfying } \theta \theta^{\prime}=\chi_p.
$$

Lemmata proved in this section are meant to hold for every field $K=\QQ(\sqrt{d})$, unless otherwise specified.

\begin{lemma}\label{irred7}
 If $d\ne -3$, then $\bar{\rho}_{E, 7}$ is irreducible.   
\end{lemma}

\begin{proof}
 Assume not and say $\bar{\rho}_{E, 7}$ is reducible. Since $E$ has full $2$-torsion over $K$, has non-trivial $2$-torsion to be more precise, the Galois representation $\bar{\rho}_{E, 7}$ gives rise to a $K$-rational point on the modular curve $X_0(28)$. The quadratic points of $X_0(28)$ have been classified by Bruin and Najman in \cite[Table 4]{bruin2015hyperelliptic}. They are all defined over $\QQ(\sqrt{-3}),\QQ(\sqrt{-7})$ or $\QQ(\sqrt{-23})$. Therefore, $X_0(28)(K)$ has no points that could arise from the Frey curve $E$.     
\end{proof}

\begin{lemma}\label{irred11}
  $\bar{\rho}_{E, 11}$ is irreducible.  
\end{lemma}

\begin{proof}
Assume not and say $\bar{\rho}_{E, 11}$ is reducible. Since $E$ has full $2$-torsion over $K$, has non-trivial $2$-torsion to be more precise, the Galois representation $\bar{\rho}_{E, 11}$ gives rise to a $K$-rational point on the modular curve $X_0(44)$. The quadratic points of $X_0(44)$ have been determined by \"{O}zman and Siksek in \cite[Table 8.4]{ozman2019quadratic}. They are all defined over $\QQ(\sqrt{-7})$. Therefore, $X_0(44)(K)$ has no points that could arise from the Frey curve $E$.  
\end{proof}

\begin{lemma}\label{irred13}
    If $d\ne -3$, then $\bar{\rho}_{E, 13}$ is irreducible.
\end{lemma}

\begin{proof}
    Let $d\ne -3$ and assume that $\bar{\rho}_{E, 13}$ is reducible. Since $E$ has full $2$-torsion over $K$, has non-trivial $2$-torsion to be more precise, the Galois representation $\bar{\rho}_{E, 13}$ gives rise to a $K$-rational point on the modular curve $X_0(52)$. The quadratic points of $X_0(52)$ have been determined by \"{O}zman and Siksek in \cite[Table 8.7]{ozman2019quadratic}. They are defined over $\QQ(i)$ or $\QQ(\sqrt{-3})$. Therefore, $X_0(52)(K)$ has no points that could arise from the Frey curve $E$.
\end{proof}

\begin{lemma}\label{irred17}
  $\bar{\rho}_{E, 17}$ is irreducible.  
\end{lemma}

\begin{proof}
    Assume not and say $\bar{\rho}_{E, 17}$ is reducible. Since $E$ has full $2$-torsion over $K$, has non-trivial $2$-torsion to be more precise, the Galois representation $\bar{\rho}_{E, 17}$ gives rise to a $K$-rational point on the modular curve $X_0(34)$. The quadratic points of $X_0(34)$ have been determined by \"{O}zman and Siksek in \cite[Table 8.1]{ozman2019quadratic}. They are defined over $\QQ(i)$, $\QQ(\sqrt{-2})$ and $\QQ(\sqrt{-15})$. Therefore, $X_0(34)(K)$ has no points that could arise from the Frey curve $E$.
\end{proof}

Next, we will discuss the irreducibility of $\bar{\rho}_{E, p}$ in a more general framework. While doing so, the following two lemmas will be needed. Note that the following lemma is a humble generalization of \cite[Lemme 1]{kraus1996courbes}, which was originally stated for semistable elliptic curves over quadratic fields. As we observe from Lemma \ref{condE}, this is usually not the case over the quadratic fields considered.

\begin{lemma}\label{krausmodified}
    Let $E$ be an elliptic curve over a quadratic field $K$. Let $S$ denote the set containing all additive primes of $E$. Let $p$ be a rational prime that is not ramified in $K$ such that $E$ does not have additive reduction at the primes above $p$. Assume that $\bar{\rho}_{E, p}$ is reducible. Denote the isogeny characters corresponding to the representation $\bar{\rho}_{E, p}$ by $\theta$ and $\theta^{\prime}$. Then, $\theta$ and $\theta^{\prime}$ are unramified away from $p$ and the primes in $S$. Moreover, if $\mathfrak{p}$ is a prime above $p$, then either $\theta$ or $\theta^{\prime}$ is unramified at  $\mathfrak{p}$.
\end{lemma}

\begin{proof}
 Let $\mathfrak{p}\notin S$ be a prime of $K$ whose residue characteristic is coprime to $p$. By Lemma \ref{Esem}, we know that $E$ is semistable away from $S$. In particular, $E$ is semistable at $\mathfrak{p}$. If $E$ has good reduction at  $\mathfrak{p}$, then \cite[Theorem VII.7.1]{silverman2009arithmetic} implies that $E[p]$ is unramified at $\mathfrak{p}$, i.e. the action of $I_{\mathfrak{p}}$ on $E[p]$ is trivial. This shows that the representation $\bar{\rho}_{E, p}$ is unramified at $\mathfrak{p}$, hence both of the characters are unramified at $\mathfrak{p}$. If $E$ has multiplicative reduction at $\mathfrak{p}$, then by the theory of Tate curves, we have either $\# \bar{\rho}_{E, p}(I_{\mathfrak{p}})=1$ or $\# \bar{\rho}_{E, p}(I_{\mathfrak{p}})=p$. In the former case, we immediately see that $\theta$ and $\theta^{\prime}$ are unramified at $\mathfrak{p}$. Now, we consider the latter case, so let us assume $\# \bar{\rho}_{E, p}(I_{\mathfrak{p}})=p$. Note that we are considering $\left.\bar{\rho}_{E, p}\right|_{I_{\mathfrak{p}}}: I_{\mathfrak{p}}\twoheadrightarrow \bar{\rho}_{E, p}(I_{\mathfrak{p}})$. Now, in this case, for any $\sigma\in I_{\mathfrak{p}}$, we have $$\bar{\rho}_{E, p}(\sigma)^p=\bar{\rho}_{E, p}(\sigma^p)\sim \left(\begin{array}{ll}
1 & 0 \\
0 & 1
\end{array}\right).$$Note that $$\bar{\rho}_{E, p}(\sigma)\sim \left(\begin{array}{ll}
\theta(\sigma) & * \\
0 & \theta^{\prime}(\sigma)
\end{array}\right).$$Without loss of generality, assume not and say $\theta(\sigma)=a\ne 1$ for some $a\in \mathbb{F}_p^{\times}$. Then,
$$\bar{\rho}_{E, p}(\sigma^p)\sim \left(\begin{array}{ll}
a^p & * \\
0 & \theta^{\prime}(\sigma)^p
\end{array}\right)\equiv_{\text{mod $p$}}\left(\begin{array}{ll}
a & * \\
0 & \theta^{\prime}(\sigma)^p
\end{array}\right)\not\sim \left(\begin{array}{ll}
1 & 0 \\
0 & 1
\end{array}\right),$$which is a contradiction. Therefore, $\theta$ and $\theta^{\prime}$ are unramified at $\mathfrak{p}$ when $\# \bar{\rho}_{E, p}(I_{\mathfrak{p}})=p$. Now, let $\mathfrak{p}\notin S$ be a prime above $p$. First, note that either $E$ has good ordinary reduction at $\mathfrak{p}$ or multiplicative reduction at $\mathfrak{p}$. By \cite[Corollary p.274 \& Corollary p.276]{serre1972proprietes}, we see that one of the characters $\left.\theta\right|_{I_{\mathfrak{p}}}$, $\left.\theta^{\prime}\right|_{I_{\mathfrak{p}}}$ is the trivial character and the other one is the mod $p$ cyclotomic character $\chi_p: I_{\mathfrak{p}}\to \mathbb{F}_p^*$. Therefore, either $\theta$ or $\theta^{\prime}$ is unramified at  $\mathfrak{p}$. 
\end{proof}

\begin{lemma}[{\cite[Lemma 6.3]{freitas2015fermat}}]\label{freitsiksek-exponent}
Let $E$ be an elliptic curve over a number field $K$ of conductor $\mathcal{N}$ and let $p \geq 5$ be a rational prime. Suppose $\bar{\rho}_{E, p}$ is reducible and write

$$
\bar{\rho}_{E, p} \sim\left(\begin{array}{cc}
\theta & * \\
0 & \theta^{\prime}
\end{array}\right)
$$where $\theta, \theta^{\prime}: G_K \rightarrow \mathbb{F}_p^*$ are characters. Write $\mathcal{N}_\theta$ and $\mathcal{N}_{\theta^{\prime}}$ for the respective conductors of these characters. Let $q$ be a prime of $K$ with $\mathfrak{q} \nmid p$.

\begin{enumerate}[(a)]
\item If $E$ has good or multiplicative reduction at $\mathfrak{q}$ then $v_{\mathfrak{q}}\left(\mathcal{N}_\theta\right)=v_{\mathfrak{q}}\left(\mathcal{N}_\theta^{\prime}\right)=0$.
\item If $E$ has additive reduction at $\mathfrak{q}$ then $v_{\mathfrak{q}}(\mathcal{N})$ is even and

$$
v_{\mathfrak{q}}\left(\mathcal{N}_\theta\right)=v_{\mathfrak{q}}\left(\mathcal{N}_{\theta^{\prime}}\right)=\frac{1}{2} v_{\mathfrak{q}}(\mathcal{N})
$$
\end{enumerate}
\end{lemma}

We may now prove that the Galois representation $\bar{\rho}_{E, p}$ in the general case. 

\begin{lemma}\label{irred19}
Assume that $p\ne |d|$ when $d=-43,\pm 19, 23$. Then, $\bar{\rho}_{E, p}$ is irreducible for $p>B^{\prime}_K$, where $$B^{\prime}_K:=\left\{\begin{array}{l}
17 \text {, if we are in any case of Theorem \ref{mainthm1} or if $d=-3,-11,-19,-43$ and $2\mid abc$} \\

8394593 \text {, if  $d=-3,-11, -19,-43$ and $2\nmid abc$. }
\end{array}\right..$$    
\end{lemma} 

\begin{proof}
We will treat all the cases together. Assume not and say $\bar{\rho}_{E, p}$ is reducible. Then, we know that $$
\bar{\rho}_{E, p} \sim\left(\begin{array}{cc}
\theta & * \\
0 & \theta^{\prime}
\end{array}\right)
$$
for some characters $\theta$, $\theta^{\prime}$ : $G_K \rightarrow \mathbb{F}_p^{\times}$ such that $\theta \theta^{\prime}=\chi_p$, where $\chi_p$ is the mod $p$ cyclotomic character given by the action of $G_K$ on the group $\mu_p$ of $p$-th roots of unity.
By Lemma \ref{krausmodified}, we know that $\theta, \theta^{\prime}$ are unramified away from $p$ and the prime $\mfp$. Now, consider the following cases:

\begin{enumerate}
    \item Suppose that $p$ is coprime to  $\mathcal{N}_\theta$ or to $\mathcal{N}_{\theta^{\prime}}$. By replacing $E$ with the $p$-isogenous curve $E/ker\theta$, if needed, we may swap the characters $\theta$ and $\theta^{\prime}$ in the matrix representation of $\bar{\rho}_{E, p}$ and assume that $(p, \mathcal{N}_\theta)=1$. This allows us to assume that $\theta$ is unramified away from the prime $\mfp$. 

    \begin{itemize}
        \item Assume that $K$ is real quadratic. We shall use Lemma \ref{freitsiksek-exponent} to determine $\mathcal{N}_\theta$. By Lemma \ref{condE}, we know that the exponent of the prime $\mfp$ in the conductor of $E$ is either $1$ or $4$. If the exponent is $1$, i.e. $E$ has multiplicative reduction at $\mfp$, Lemma \ref{freitsiksek-exponent} implies that $v_{\mfp}(\mathcal{N}_{\theta})=0$. If the exponent is $4$, i.e. $E$ has additive reduction at $\mfp$, Lemma \ref{freitsiksek-exponent} implies that $v_{\mfp}(\mathcal{N}_{\theta})=v_{\mfp}(N_E)/2=2$. Therefore, either $\mathcal{N}_{\theta}=1$ or $\mathcal{N}_{\theta}=\mfp^2$. Denote the two real places of $K$ by $\infty_1,\infty_2$. It follows that $\theta$ is a character of the ray class group for the modulus $$\infty_1\infty_2\text{ or }\mfp^2\infty_1\infty_2.$$Using Magma, we see that the ray class group is one of $\{1\}$, $\ZZ/2\ZZ$  or $\ZZ/2\ZZ\times \ZZ/2\ZZ$ in all cases. Therefore, the order of $\theta$ is $1$ or $2$. If the order of  $\theta$ is $1$, then $\theta$ is the trivial character. Therefore $E$ has a point of order $p$ over $K$. But in \cite[Theorem 3.1]{kamienny1992torsion}, Kamienny showed that the order of such a torsion point is at most $13$. If $\theta$ has order $2$, then $E$ has a $p$-torsion point defined over a quadratic extension $K$, say $L$. Then, $[L:\QQ]=4$. According to Kamienny et al., \cite{derickx2023torsion}, the possible prime torsions of elliptic curves over number fields of degree $4$ imply that $p\le17$. In each case, we derive a contradiction as we assumed $p\ge 19$.

    \item Assume that $K$ is imaginary quadratic. First, let us observe that the characters $\theta^2$ and $\theta'^2$ are unramified at $\mfp$: Let $D_{\mfp}\subseteq G_K$ be the decomposition group at $\mfp$. By Lemma \ref{potmultb}, $E$ has potentially multiplicative reduction at $\mfp$. Then,
        \begin{equation*}
    \left.\bar{\rho}_{E, p}\right|_{D_{\mfp}}\sim\phi\cdot\left(\begin{array}{ll}
\chi_p & * \\
0 & 1
\end{array}\right)\text{,}
\end{equation*}where $\phi$ is at worst a quadratic character. In particular, both $\theta^2$ and $\theta'^2$ are unramified at $\mfp$. Note that $\theta$ is unramified away from $\mfp$ by assumption. Therefore, $\theta^2$ is an unramified character. Since the class group of $K$ is trivial, $\theta^2$ is the trivial character. If the character $\theta$ is itself trivial, then $E$ has a $K$-point of order $p$. If $\theta^2$ is quadratic, then $E$ has a $p$-torsion point defined over a quadratic extension of $K$. It follows that the Frey curve $E$ or its twist by $\theta$ has a $K$-rational point of order $p$. Again by a result of Kamienny, \cite[Theorem 3.1]{kamienny1992torsion}, we have a contradiction.
    \end{itemize}

    \item Now, assume that $p$ is not coprime to $\mathcal{N}_\theta$ nor to $\mathcal{N}_{\theta^{\prime}}$. Observe that $p$ is not ramified in $K$ in each case, as we assumed $p\ne |d|$ when necessary. If $p$ is inert in $K$, then by Lemma \ref{krausmodified}, we see that $p$ divides precisely one of the conductors, which is a contradiction. So let us assume that $p$ splits in $K$, say $p\mathcal{O}_K=\mathfrak{p}_1 \mathfrak{p}_2$. Since $\mathfrak{p}_1, \mathfrak{p}_2$ are unramified, and $\bar{\rho}_{E, p}$ is reducible; by the contrapositive of \cite[Corollary 6.2]{freitas2015fermat}, we see that $E$ cannot have good supersingular reduction at $\mathfrak{p}_1$ and $\mathfrak{p}_2$. But $E$ is semistable away from $\mfp$, hence $E$ has good ordinary or multiplicative reduction at these primes. It follows from \cite[Proposition 6.1(ii)]{freitas2015fermat} that
\begin{equation}\label{prop6.1}
    \left.\bar{\rho}_{E, p}\right|_{I_{\mathfrak{p_1}}} \sim\left(\begin{array}{ll}
\chi_p & * \\
0 & 1
\end{array}\right)\text{,}
\end{equation}which shows that exactly one of the isogeny characters is ramified at $\mathfrak{p}_1$, and the other one is ramified at $\mathfrak{p}_2$. We can assume that $\mathfrak{p}_1 \mid \mathcal{N}_\theta, \mathfrak{p}_1 \nmid \mathcal{N}_{\theta^{\prime}}$ and $\mathfrak{p}_2 \mid \mathcal{N}_{\theta^{\prime}}$, $\mathfrak{p}_2 \nmid \mathcal{N}_\theta$. By (\ref{prop6.1}), we see that $\left.\theta\right|_{I_{\mathfrak{p}_1}}=\left.\chi_p\right|_{I_{\mathfrak{p}_1}}$ and $\left.\theta^{\prime}\right|_{I_{\mathfrak{p}_2}}=\left.\chi_p\right|_{I_{\mathfrak{p}_2}}$. Now, we consider the following two cases:

\begin{itemize}
    \item First, assume that we are in any case of Theorem \ref{mainthm1} (i.e $K$ is real quadratic) or $d=-3,-11,-19,-43$ and $\mfp\mid abc$. Lemma \ref{potmultb} shows that we are now dealing with the cases where the Frey curve $E/K$ has potentially multiplicative reduction at $\mfp$. Since there is only one bad prime for $E$ and it is of potentially multiplicative reduction, $\theta^2$ is unramified everywhere except $\mathfrak{p}_1$. (This follows from the fact that since $E/K$ has potentially multiplicative reduction at $\mfp$, the restriction $\left.\bar{\rho}_{E, p}\right|_{I_{\mfp}}$ is up to semi-simplification equal to $\phi \oplus \phi\cdot\chi_p$, where $\phi$ is at worst a quadratic character.) We also know that $\left.\theta^2\right|_{I_{\mathfrak{p}_1}}=\left.\chi_p^2\right|_{I_{\mathfrak{p}_1}}$. By \cite[Lemma 4.3]{cturcacs2018fermat}, we have that
$$
\theta^2\left(\sigma_{\mfp}\right) \equiv N_{K_{\mathfrak{p}_1} / \mathbb{Q}_p}\left(\iota_{\mathfrak{p}_1}(\mfp)\right)^2 \pmod p,
$$where $\sigma_{\mfp}$ is the Frobenius element at $\mfp$, and $\iota_{\mathfrak{p}_1}$ is the inclusion map from $K$ into the completion of $K$ with respect to $\mathfrak{p}_1$. Also, by \cite[Lemma 6.3]{csengun2018asymptotic}, we may assume that $\theta^2\left(\sigma_{\mfp}\right) \equiv 1\pmod p$. Let $b_K$ be a generator of the prime $\mfp$ in $K$. We have that
$$
N_{K_{\mathfrak{p}_1} / \mathbb{Q}_p}\left(\iota_{\mathfrak{p}_1}(\mfp)\right)^2-1=N_{K_{\mathfrak{p}_1} / \mathbb{Q}_p}\left(b_K\right)^2-1
=3,
$$which implies that $p\mid3$, a contradiction as $p> 5$ in each case.
    \item Let us now assume $K$ is imaginary quadratic and $\mfp\nmid abc$. Note that in this case, the above argument fails since $E$ has potentially good reduction at its bad prime $\mfp$ by Lemma \ref{lemma: Epotgoodred}. By Lemma \ref{lemma: Epotgoodred}, the elliptic curve $E$ has potentially good reduction at $\mathfrak{P}$. Since the residue characteristic at $\mathfrak{P}$ is $2$, \cite[Section 1, Case $p=2$]{Kraus1990} implies that the image of inertia $\bar{\rho}_{E, p}\left(I_{\mathfrak{P}}\right)$ has order dividing $24$. In particular, $\theta^{24}$ is unramified at $\mathfrak{P}$. It follows that $\left.\theta^{24}\right|_{I_{\mathfrak{p}_1}}=\left.\chi_p^{24}\right|_{I_{\mathfrak{p}_1}}$ is unramified away from $\mathfrak{p}_1$. By \cite[Lemma 4.3]{cturcacs2018fermat}, $\theta^{24}\left(\sigma_{\lambda}\right) \equiv \operatorname{Norm}_{K_{\mathfrak{p}_1} / \mathbb{Q}_p}(\alpha)^2=4 \pmod p$, where $\alpha$ is any nonzero prime-to-$p$ element of $K$ and $\sigma_{\lambda}$ is the Frobenius automorphism at $\lambda=\langle\alpha\rangle$. Therefore, the polynomial $x^{24}-4$ has a root $\theta\left(\sigma_{\lambda}\right)$ modulo $p$. Note that this congruence holds for any nonzero prime-to-$p$ element of $K$ and the ideal $\lambda=\langle\alpha\rangle$. Let $b_K$ be a generator of the prime $\mfp$ in $K$. In particular, the congruence holds for the prime-to-$p$ element $b_K$ and the ideal $\mfp$. On the other hand, recall that $E$ has potentially good reduction at $\mfp$ since $v_\mfp\left(j_E\right)>$ 0. Let $P_\mfp(x)$ be the characteristic polynomial of the Frobenius of $E$ at $\mfp$. Since $\rho_{E,p} \sim\left(\begin{array}{cc}\theta & * \\ 0 & \theta^{\prime}\end{array}\right)$, we get $P_\mfp(x) \equiv\left(x-\theta\left(\sigma_\mfp\right)\right)\left(x-\theta^{\prime}\left(\sigma_\mfp\right)\right)\pmod p$. Therefore, we deduce that $p\mid \operatorname{Res}(x^{24}-4, P_\mfp(x))$, where $\operatorname{Res}$ denotes the resultant of the polynomials. By \cite[Proposition 1.6]{david2011caract} we know that $P_\mfp(x)\in\ZZ[x]$ and the absolute value of its roots is at most $\sqrt{\Norm(\mfp)}=2$. The possibilities for $P_\mfp(x)$ are as follows:
    \begin{eqnarray*}
        P_1&=&x^2+4\\
        P_2&=&x^2\pm x+4\\
        P_3&=&x^2\pm 2x+4\\
        P_4&=&x^2\pm 3x+4\\
        P_5&=&x^2\pm 4x+4\\
        \end{eqnarray*}By computing the resultants, we see that none of them has prime divisors greater than $8394593$. Since $p>8394593$, we have a contradiction.
    
\end{itemize}
\end{enumerate}
\end{proof}

The following lemma deals with the cases where we had to assume $p\ne |d|$ in the statement of Lemma \ref{irred19}.

\begin{lemma}\label{lemma:irred,d=23}
$\bar{\rho}_{E, |d|}$ is irreducible.
\end{lemma}

\begin{proof}
    We will proceed as in the proof of Lemma \ref{irred19}. Assume not and say $\bar{\rho}_{E, |d|}$ is reducible. This yields $$
\bar{\rho}_{E, |d|} \sim\left(\begin{array}{cc}
\theta & * \\
0 & \theta^{\prime}
\end{array}\right)
$$where $\theta,\theta^{\prime}$ are isogeny characters as in Lemma \ref{irred19}. By Lemma \ref{krausmodified}, we know that $\theta,\theta^{\prime}$ are unramified away from $|d|$ and the prime $\mfp$ and we consider the following cases:
\begin{itemize}
    \item When $|d|$ is coprime to one of the conductors $\mathcal{N}_\theta$, $\mathcal{N}_\theta^{\prime}$, the contradiction follows as in the proof of Lemma \ref{irred19}.
    \item In the proof of Lemma \ref{irred19}, recall that we made use of the fact that the prime $p$ is not ramified in the number fields  when $p$ is not coprime to $\mathcal{N}_\theta$ nor to $\mathcal{N}_\theta^{\prime}$. Since this is not the case here, we need a different approach. Let $\mathfrak{D}$ be the unique prime above $d$. By \cite[Proposition 6.1]{freitas2015fermat}, we have $$ \left.\bar{\rho}_{E, |d|}\right|_{I_{\mathfrak{D}}} \sim\left(\begin{array}{ll}
\chi_{|d|} & * \\
0 & 1
\end{array}\right),$$where $\chi_{|d|}$ is the mod $|d|$ cyclotomic character. Therefore, precisely one of $\theta, \theta^{\prime}$ is unramified at $\mathfrak{D}$. It follows that precisely one of $\mathcal{N}_\theta^{\prime}$ is coprime to $|d|$, which is a contradiction.
\end{itemize}
\end{proof}

In the following corollary, we summarize the cases where we have irreducibility.

\begin{corollary}\label{bound:B_K}
$\bar{\rho}_{E, p}$ is irreducible for $p>B_K$, where $$B_K:=\left\{\begin{array}{l}
5 \text {, if we are in any case of Theorem \ref{mainthm1} or Theorem \ref{mainthm3} with $2\mid abc$} \\
13 \text {, if we are in the case of Theorem \ref{mainthm2:d=-3}} \\
8394593 \text {, if we are in any case of Theorem \ref{mainthm3} with $2\nmid abc$. }
\end{array}\right..$$       
\end{corollary}

\subsection{Absolute irreducibility}
Over totally real fields, the irreducibility of the representation $\bar{\rho}_{E, p}$ immediately implies the absolute irreducibility. However, this is not the case when the considered number field is totally complex. Recall that in the statement of Theorem \ref{mainthm3}, we assumed Serre's modularity conjecture when considering the $d$-Fermat equation over imaginary quadratic fields $K$. Therefore, we want the mod $p$ representation $\bar{\rho}_{E, p}$ attached to the Frey curve $E/K$ to satisfy the hypothesis of the conjecture. Hence, we need absolute irreducibility of the representation $\bar{\rho}_{E, p}$ in the setting of Theorem \ref{mainthm3} and we will obtain it by proving the surjectivity of the Galois representation $\bar{\rho}_{E, p}$. We will make use of the following two well-known results to achieve this:

\begin{lemma}[{\cite[Lemma 3.4]{freitas2015asymptotic}}]\label{divinertia}
Let $E$ be an elliptic curve over $K$ with $j$-invariant $j$. Let $p \geq 5$ and let $\mathfrak{q} \nmid p$ be a prime of $K$. Then $p \mid \# \bar{\rho}_{E, p}\left(I_{\mathfrak{q}}\right)$ if and only if $E$ has potentially multiplicative reduction at $\mathfrak{q}$ (i.e. $v_{\mathfrak{q}}(j)<0$ ) and $p \nmid v_{\mathfrak{q}}(j)$.    
\end{lemma}

Observe that Lemma \ref{divinertia} applies to our setting. Indeed, by the proof of Lemma \ref{potmultb}, we know that the elliptic curve $E$ is semistable at $\mfp$ and we have $$v_{\mfp}(j_E)=8v_{\mfp}(2)-2pv_{\mfp}(b),$$which is clearly not divisible by $p$. Therefore, we deduce that $p \mid \# \bar{\rho}_{E, p}(I_{\mfp})$.

\begin{lemma}[{\cite[Lemma 2, p.12]{swinnerton1973modular}}]\label{irredsubg}
Let $E$ be an elliptic curve over a number field $K$. Denote the image of the mod $p$ representation attached to $E$ by $\bar{\rho}_{E, p}(G_K)$, which is a subgroup of $\operatorname{GL}_2(\mathbb{F}_p)$. If p divides $\#\bar{\rho}_{E, p}(G_K)$, then either $\bar{\rho}_{E, p}$ is reducible or $\bar{\rho}_{E, p}(G_K)$ contains $\operatorname{SL}_2(\mathbb{F}_p)$.
\end{lemma}

We are now ready to discuss the surjectivity of $\bar{\rho}_{E, p}$ attached to the elliptic curve $E/\QQ(\sqrt{d})$ when $d=-3,-11,-19,-43$:

\begin{theorem}\label{surj}
Assume $p>5$ and $p\ne |d|$ when $d=-11,-19,-43$ and $2\mid abc$.
Assume $p>13$ when $K=\QQ(\sqrt{-3})$ and $2\mid abc$. Then, the Galois representation $\bar{\rho}_{E, p}$ is surjective.
\end{theorem}

\begin{proof}
   Note that by Lemma \ref{divinertia}, we know that for any prime $\mfp$ lying over $2$, we have $p \mid \# \bar{\rho}_{E, p}(I_{\mfp})$. Then $p$ divides the order of the image $\bar{\rho}_{E, p}(G_K)$. By Lemmas \ref{irred7}--\ref{irred17} and Lemma \ref{irred19}, we know that $\bar{\rho}_{E, p}$ is irreducible for the specified bounds $p$. By Lemma \ref{irredsubg}, $\Ker(\det)=\operatorname{SL}_2(\mathbb{F}_p)\subseteq \bar{\rho}_{E, p}(G_K)=\im(\bar{\rho}_{E, p})$. Note that the mod $p$ cyclotomic character $\chi_p$ is surjective when $K \cap \QQ(\zeta_p)=\QQ$, i.e. when $p$ is not ramified in $K$. Since the composition $\det(\bar{\rho}_{E, p})$=$\chi_p$ is surjective and $\Ker(\det)=\operatorname{SL}_2(\mathbb{F}_p)\subseteq \bar{\rho}_{E, p}(G_K)=\im(\bar{\rho}_{E, p})$, we deduce that $\bar{\rho}_{E, p}$ is surjective as well.
\end{proof}

\begin{theorem}
    Let $K=\QQ(\sqrt{-3}), \QQ(\sqrt{-11}), \QQ(\sqrt{-19}), \QQ(\sqrt{-43})$. Assume that $2\nmid abc$ and $p>M_K$ for some effectively computable constant $M_K$ depending only on $K$. Assume in addition that $p$ splits in $K$ and $p \equiv 3\pmod 4$. If $\bar{\rho}_{E, p}$ is irreducible, then it is absolutely irreducible. 
\end{theorem}

\begin{proof}
    We will proceed as in the proof of \cite[Theorem 1.5]{najman2021irreducibility}. Assume $\mfp\nmid abc$. Assume that $\bar{\rho}_{E, p}$ is irreducible but absolutely reducible. In this case, the image $\bar{\rho}_{E, p}(G_K)$ is contained in a nonsplit Cartan subgroup. We have that $\bar{\rho}_{E, p} \otimes_{\mathbb{F}_p} \overline{\mathbb{F}}_p \sim\left(\begin{array}{cc}\lambda & 0 \\ 0 & \lambda^p\end{array}\right)$, where $\lambda: G_K \rightarrow \mathbb{F}_{p^2}^{\times}$ is a character. The latter is not $\mathbb{F}_p$-valued and $\lambda^{p+1}=\chi_p$, where $\chi_p: G_K \rightarrow \mathbb{F}_p^{\times}$is the $\bmod $ $p$ cyclotomic character.

    Assume that $p$ splits in $K$ and $p \equiv 3\pmod 4$. When the prime $p$ is greater than the constant given in the theorem, the first statement of Theorem \ref{thm: larson} gives CM elliptic curve $E^{\prime} / K$ satisfying $\bar{\rho}_{E^{\prime}, p} \otimes \overline{\mathbb{F}}_p \sim\left(\begin{array}{cc}\theta & 0 \\ 0 & \theta^{\prime}\end{array}\right)$ such that $\theta^{12}=\lambda^{12}$. Since $p$ splits in $K$, we see that the image $\bar{\rho}_{E^{\prime}, p}\left(G_K\right)$ is contained inside a split Cartan subgroup, which implies that the character $\theta$ is in fact $\mathbb{F}_p$-valued. In particular, the order of $\theta$ is divisible by $p-1$. Theorem \ref{thm: larson} also implies that $\lambda \theta^{-1}$ is unramified away from the additive primes of $E$. Therefore, $\lambda \theta^{-1}$ is unramified at $\mathfrak{p} \mid p$. Since $\lambda^{p+1}=\chi_p$, we get $\left.\theta^{p+1}\right|_{I_{\mathfrak{p}}}=\left.\chi_p\right|_{I_{\mathfrak{p}}}$. Note that $p \equiv 3\pmod 4$, hence $\frac{p-1}{2}$ is odd. We deduce that $\left.\chi_p^{(p-1) / 2}\right|_{I_{\mathfrak{p}}}=\left.\left(\theta^{p-1}\right)^{(p+1) / 2}\right|_{I_{\mathfrak{p}}}=1$. However, the order of $\left.\chi_p^{(p-1) / 2}\right|_{I_{\mathfrak{p}}}$ is $p-1$ since it surjects on $\mathbb{F}_p^\times$, contradicting the fact that $\left.\chi_p^{(p-1) / 2}\right|_{I_{\mathfrak{p}}}=1$. We proved that when $p$ splits in $K$ and $p \equiv 3\pmod 4$, if the representation is $\bar{\rho}_{E, p}$ is irreducible, then it must be absolutely irreducible.
\end{proof}

\begin{theorem}[{\cite[Theorem 1]{larson2014determinants}}]\label{thm: larson}
Let $K$ be a number field. There exists a finite set of primes $M_K$, depending only on $K$, such that, for any prime $p \notin M_K$ and any elliptic curve $E / K$ for which $\bar{\rho}_{E, p} \otimes \overline{\mathbb{F}}_p$ is conjugate to $\left(\begin{array}{cc}\lambda & * \\ 0 & \lambda^{\prime}\end{array}\right)$ where $\lambda, \lambda^{\prime}: G \rightarrow \mathbb{F}_p^{\times}$are characters, one of the following happens.
\begin{enumerate}
    \item There exists an elliptic curve $E^{\prime} / K$ with $C M$, whose CM field is contained in $K$, with $\bar{\rho}_{E^{\prime}, p} \otimes \overline{\mathbb{F}}_p$ is conjugate to $\left(\begin{array}{cc}\theta & * \\ 0 & \theta^{\prime}\end{array}\right)$ and such that $\theta^{12}=\lambda^{12}$.
    \item The Generalized Riemann Hypothesis (GRH) fails for $K=\mathbb{Q}(\sqrt{-p})$ and $\theta^{12}=\chi_p^6$. Moreover, in this case, $\bar{\rho}_{E^{\prime}, p}$ is already reducible over $\mathbb{F}_p$ and $p \equiv 3\pmod 4$.
\end{enumerate}
\end{theorem}

\subsection{Level-lowering: Eliminating newforms}\label{section:level-lowering}

In order to prove Theorem \ref{mainthm1}, one needs to apply the level-lowering result Theorem $\ref{levellow}$. Observe that every condition of Theorem $\ref{levellow}$ is now satisfied. Indeed, by Theorem \ref{box}, we know that $E/K$ is modular. The conditions $(i)$ and $(ii)$ of Theorem \ref{levellow} are satisfied by Lemma \ref{Esem}. The condition $(iii)$ is automatically satisfied since we are dealing with quadratic fields $K$ and $p>5$ in every case. Finally, the representation $\bar{\rho}_{E, p}$ is irreducible for the given bounds on $p$ as we explained in Section \ref{section:irreducibility}. 

By Theorem \ref{levellow}, we know that there is a Hilbert newform $\mathfrak{f}$ of parallel weight $2$ and level $\mathcal{N}$ as indicated in Table \ref{loweredleveltable}, a prime $\mathfrak{p}$ above $p$ in the Hecke eigenvalue field $\QQ_{\mathfrak{f}}$ of $\mathfrak{f}$ such that $\bar{\rho}_{E, p} \sim \bar{\rho}_{\mathfrak{f}, \mathfrak{p}}$.

\begin{table}
$$\begin{array}{|c|c|c|c|}
\hline d & \text{Level  $\mathcal{N}$}  &\text{Condition on $i$} &\text{Ideal $\mathfrak{D}$}\\
\hline 3 & \mfp^i\mathfrak{D}  & i\in\{1,4\}  & \mathfrak{D}\mid 3\\
\hline 5 & \mfp^i\mathfrak{D} & i\in\{1,4\} & \mathfrak{D}\mid5 \\
\hline 7 & \mfp^i\mathfrak{D}  & i\in\{1,4\} & \mathfrak{D}\mid7\\
\hline 11 & \mfp^i\mathfrak{D}  & i\in\{1,4\}  & \mathfrak{D}\mid11\\
\hline 13 & \mfp^i\mathfrak{D} & i=1  & \mathfrak{D}\mid13\\
\hline 19 & \mfp^i\mathfrak{D} & i\in\{1,4\} & \mathfrak{D}\mid19 \\
\hline 23 & \mfp^i\mathfrak{D} & i\in\{1,4\}& \mathfrak{D}\mid23  \\
\hline
\end{array}$$
\caption{Lowered level $\mathcal{N}$}
\label{loweredleveltable}
\end{table}

\begin{remark}
    Similar to our discussion in the beginning of Chapter \ref{galoisrepr} for imaginary quadratic fields, the ideal $\mathcal{N}$ consists of only powers of the prime $\mfp$ if $p\mid r$, and it is divisible by $\mfp$ and $\mathfrak{D}$ otherwise. Since the level of the former case was already treated in \cite{freitas2015fermat} and \cite{michaud2021fermat}, we assume $\gcd(p,r)>1$.
\end{remark}

We will now discuss how to eliminate the arising Hilbert newforms.

\subsubsection{Existing techniques}\label{techniques}

    We start by stating a theorem allowing us to bound the exponent $p$. This is \cite[Lemma 7.1]{freitas2015fermat}:
    \begin{lemma}\label{C_f}
      Let $K$ be a totally real field, and let $p \geq 5$ be a prime. Let $E$ be an elliptic curve over $K$ of conductor $\mathcal{N}$, and let $\mathfrak{f}$ be a newform of parallel weight 2 and level $\mathcal{N}_p$. Let $t$ be a positive integer satisfying $t \mid \# E(K)_{\text {tors }}$. Let $\mathfrak{q} \nmid t \mathcal{N}_p$ be a prime ideal of $\mathcal{O}_K$ and let

$$
\mathcal{A}_{\mathfrak{q}}=\{a \in \mathbb{Z}: \quad|a| \leq 2 \sqrt{\operatorname{Norm}(\mathfrak{q})}, \quad \operatorname{Norm}(\mathfrak{q})+1-a \equiv 0 \pmod{t}\}
$$

If $\bar{\rho}_{E, p} \sim \bar{\rho}_{\mathrm{f}, \varpi}$ then $\varpi$ divides the principal ideal

$$
B_{\mathfrak{f}, \mathfrak{q}}=\operatorname{Norm}(\mathfrak{q})\left((\operatorname{Norm}(\mathfrak{q})+1)^2-a_{\mathfrak{q}}(\mathfrak{f})^2\right) \prod_{a \in \mathcal{A}_{\mathfrak{q}}}\left(a-a_{\mathfrak{q}}(\mathfrak{f})\right) \cdot \mathcal{O}_{\mathbb{Q}_{\mathfrak{f}}}
$$

    \end{lemma}
First, let us explain why the above lemma is useful: Set $t=4$ and $\mathcal{N}_p=\mathcal{N}$ as in Table \ref{loweredleveltable}. Let $B_{\mathfrak{f}}:=\sum_{\mathfrak{q} \in T} B_{\mathfrak{f}, \mathfrak{q}}$, where $T$ is the set of primes $\mathfrak{q}\ne \mfp,\mathfrak{D}$ of $K$ having norm less than $50$. Let $C_{\mathfrak{f}}=\Norm_{\QQ_\mathfrak{f}/\QQ}(B_{\mathfrak{f}})$. If $\bar{\rho}_{E, p} \sim \bar{\rho}_{\mathfrak{f}, \mathfrak{p}}$, previous lemma implies $\mathfrak{p}\mid B_{\mathfrak{f}}$ so that $p\mid C_{\mathfrak{f}}$. Therefore, if $p\nmid C_{\mathfrak{f}}$, then the relation $\bar{\rho}_{E, p} \sim \bar{\rho}_{\mathfrak{f}, \mathfrak{p}}$ cannot hold. 

Notice that it is not possible to eliminate a Hilbert newform $\mathfrak{f}$ using Lemma \ref{C_f} if it satisfies $C_\mathfrak{f}=0$. Therefore, we need a different approach when this is the case. The following lemma proposes another technique to eliminate such forms. We remark that the statement of the lemma below is independent of the fact that the base field $K$ is totally real or not, i.e., we are allowed to apply it when proving Theorem \ref{mainthm3}.

\begin{lemma}[Inertia Argument]\label{lemma:imageofinertia}
Let $K$ be a number field and let $E,E'$ be two elliptic curves over $K$. Let $\mathfrak{q}$ be a prime of $\mcalO_K$ and let $p\ge5$ be a rational prime. Suppose that the elliptic curve $E$ has potentially multiplicative reduction at $\mathfrak{q}$ and $p\nmid v_{\mathfrak{q}}(j)$, where $j$ denotes the $j$-invariant of the elliptic curve $E$. Suppose also that $E'$ has potentially good reduction at $\mathfrak{q}$. Then, $\bar{\rho}_{E, p} \not\sim \bar{\rho}_{E', p}$.
\end{lemma}

\begin{proof}
    The possibilities of the group $\bar{\rho}_{E', p}\left(I_{\mathfrak{q}}\right)$ has been determined by Kraus in \cite{Kraus1990}. In particular, $\# \bar{\rho}_{E, p}\left(I_{\mathfrak{q}}\right)\mid 24$. On the other hand, since $E/K$ has potentially multiplicative reduction at $\mathfrak{q}$ and $p\nmid v_{\mathfrak{q}}(j_E)$, Lemma \ref{divinertia} implies $p\mid \# \bar{\rho}_{E, p}\left(I_{\mathfrak{q}}\right)$. Therefore, $p\mid 24$. Since we assumed $p\ge 5$, this is not possible.
\end{proof}

Let us discuss how Lemma \ref{lemma:imageofinertia} is useful in our setting. Recall that the Frey curve $E$ has potentially multiplicative reduction at $\mfp$ by Lemma \ref{potmultb} and $p\nmid 8v_{\mfp}(2)-2pv_{\mfp}(b)=v_\mfp(j_E)$. Now, let $\mathfrak{f}$ be a Hilbert newform or a Bianchi modular form, depending on whether we study the $d$-Fermat equation over a totally real field or a totally complex field. Suppose that $\QQ_\mathfrak{f}=\QQ$ and $C_\mathfrak{f}=0$. Let $E'$ be an elliptic curve associated to $\mathfrak{f}$. If $E'$ has potentially good reduction at $\mfp$, i.e $v_\mfp(j_E)\ge 0$, then Lemma \ref{lemma:imageofinertia} discards the isomorphism $\bar{\rho}_{E, p} \sim \bar{\rho}_{E', p}$. Therefore, for a Hilbert newform $\mathfrak{f}$ satisfying $C_\mathfrak{f}=0$, it suffices to find an elliptic curve corresponding to $\mathfrak{f}$ having potentially good reduction at $\mfp$. Note that in the case of totally complex fields, this is not enough since it is conjecturally possible that a Bianchi modular form $\mathfrak{f}$ satisfying $\QQ_\mathfrak{f}=\QQ$ and $C_\mathfrak{f}=0$ may also correspond to a fake elliptic curve. Therefore, we need to discard this additional possibility when dealing with Bianchi modular forms. We will see an example of this phenomenon when $K=\QQ(\sqrt{-3})$. In fact, this is the only reason we had to assume Conjecture \ref{fakellcurve} in this case.

We also remark that it is not wise to try applying the machinery of Lemma \ref{lemma:imageofinertia} by choosing an odd prime in our setting because the condition $p\nmid v_{\mathfrak{q}}(j_E)$ of Lemma \ref{lemma:imageofinertia} does not hold when $\mathfrak{q}$ is an odd prime in $K$.

\subsection{Proof of Theorem \ref{mainthm1}} 

In this section, we prove Theorem \ref{mainthm1}. Recall that the proof of our main result over real quadratic fields is built upon Theorem \ref{levellow}, namely the Level-Lowering Theorem, whose conditions are satisfied for each real quadratic field $\QQ(\sqrt{d})$ as we explained at the beginning of Section \ref{section:level-lowering}. Theorem \ref{levellow} therefore guarantees the existence of a Hilbert newform $\mathfrak{f}$ of parallel weight $2$ and level $\mathcal{N}$ as indicated in Table \ref{loweredleveltable}, a prime $\mathfrak{p}$ above $p$ in the Hecke eigenvalue field $\QQ_{\mathfrak{f}}$ of $\mathfrak{f}$ such that $\bar{\rho}_{E, p} \sim \bar{\rho}_{\mathfrak{f}, \mathfrak{p}}$. In order to finish proving Theorem \ref{mainthm1}, we now show that the Hilbert newforms arising from the level-lowering theorem can be eliminated, i.e the isomorphism $\bar{\rho}_{E, p} \sim \bar{\rho}_{\mathfrak{f}, \mathfrak{p}}$ can be discarded for each possible case by using the techniques that we had introduced in the previous section.

\begin{itemize}
    \item When $K=\QQ(\sqrt{3})$, we see by using Magma that there are no Hilbert newforms at the levels $\mfp\mathfrak{D}$ and $\mfp^4\mathfrak{D}$. Therefore, Theorem \ref{mainthm1} follows immediately when $d=3$.
    
    \item When $K=\QQ(\sqrt{5})$, there are no Hilbert newforms at the level $\mfp\mathfrak{D}$. There are $12$ Hilbert newforms at the level $\mfp^4\mathfrak{D}$. Every newform $\mathfrak{f}$ at this level satisfy $C_\mathfrak{f}=0$. Therefore, we cannot eliminate those forms directly using Lemma \ref{C_f}. However, we can still try if Lemma \ref{lemma:imageofinertia} is applicable. Using Magma, all forms satisfy $\QQ_{\mathfrak{f}}=\QQ$. The elliptic curves in the isogeny classes given in the LMFDB label $2.2.5.1-1280.1-a,1280.1-b,1280.1-c,1280.1-d,1280.1-e,1280.1-f,1280.1-g,1280.1-h,1280.1-i,1280.1-j,1280.1-k,1280.1-l$ correspond to those newforms. These elliptic curves satisfy $v_{\mfp}(j)\ge0$ i.e, have potentially good reduction at $\mfp$, whereas our Frey curve $E$ has potentially multiplicative reduction at $\mfp$. By Lemma \ref{lemma:imageofinertia}, we discard all isomorphisms $\bar{\rho}_{E, p} \sim \bar{\rho}_{\mathfrak{f}, \mathfrak{p}}$, i.e we eliminated all Hilbert newforms at the level $\mfp^4\mathfrak{D}$. Theorem \ref{mainthm1} follows when $d=5$.
    
    \item When $K=\QQ(\sqrt{7})$, there are $2$ Hilbert newforms $\mathfrak{f}$ at the level $\mfp\mathfrak{D}$, for which $C_\mathfrak{f}$ is divisible by $2$ and $3$. There are $6$ Hilbert newforms at the level $\mfp^4\mathfrak{D}$. For each form $\mathfrak{f}$ at this level, $C_\mathfrak{f}$ is divisible by $2$ or $3$.

    \item When $K=\QQ(\sqrt{11})$, there are $2$ Hilbert newforms at the level $\mfp\mathfrak{D}$. $C_\mathfrak{f}$ is divisible by $3,5$ and $7$ for both of the forms. There are $12$ Hilbert newforms at the level $\mfp^4\mathfrak{D}$, for which $C_\mathfrak{f}$ is divisible by $3$ or $5$ or the ideal $B_\mathfrak{f}$ is the whole $\mcalO_K$.

    \item When $K=\QQ(\sqrt{13})$, there are $2$ Hilbert newforms at the level $\mfp\mathfrak{D}$. $C_\mathfrak{f}$ is divisible by $3,5$ or $7$ for both of the forms. 

    \item When $K=\QQ(\sqrt{19})$, there are $10$ Hilbert newforms at the level $\mfp\mathfrak{D}$. $C_\mathfrak{f}$ is divisible by $3,5$ or $19$ for these forms. There are $18$ Hilbert newforms at the level $\mfp^4\mathfrak{D}$, for which $C_\mathfrak{f}$ is divisible by $3$ or $5$ or we have $B_\mathfrak{f}=\mcalO_K$. 

    \item When $K=\QQ(\sqrt{23})$, there are $8$ Hilbert newforms at the level $\mfp\mathfrak{D}$. $C_\mathfrak{f}$ is divisible by $2,3,5$ or $11$ for these forms. There are $18$ Hilbert newforms at the level $\mfp^4\mathfrak{D}$, for which $C_\mathfrak{f}$ is divisible by $2,3$ or $5$ or the ideal $B_\mathfrak{f}$ is the whole $\mcalO_K$. 
\end{itemize}

\begin{remark}
    We finalize discussing Theorem \ref{mainthm1} with the following remark on the $d$--Fermat when $d=11,13,23$. In these cases, recall that we had to assume $p\ne d$ and it is precisely due to the item (\ref{item:ii-levellowering}) of the level-lowering theorem. Indeed, as we discussed, $E/\QQ(\sqrt{d})$ is semistable away from $\mfp$ by Lemma \ref{Esem}. From Lemma \ref{lemma:irred,d=23}, observe that $\bar{\rho}_{E, d}$ is irreducible. The condition on the ramification index is automatically satisfied since $K$ is a quadratic field. However, since we assumed $p\nmid r$, it follows from the proof of Lemma \ref{pdivdel} that the condition $p\mid v_{\mathfrak{D}}(\Delta_E)$ fails when $\mathfrak{D}\mid d$. Therefore, the level-lowering theorem cannot be applied when $p=d$, and we had to exclude this case.
\end{remark}

\section{Imaginary quadratic case}\label{section:imaginaryquadratic}

\subsection{Application of Serre's modularity conjecture}

Let us ensure that the Galois representation $\bar{\rho}_{E, p}$ satisfies the conditions of Conjecture \ref{serrconj}: First, note that our mod $p$ representation $\bar{\rho}_{E, p}: G_K\to \operatorname{GL}_2(\mathbb{F}_p)\xhookrightarrow{}\operatorname{GL}_2(\overline{\mathbb{F}}_p)$ is irreducible and surjective, hence is absolutely irreducible. It is also continuous with respect to the Krull topology on $G_K$ and the discrete topology on $\operatorname{GL}_2(\mathbb{F}_p)$. Moreover, our bounds on $p$ given in the statement of Theorem \ref{mainthm3} guarantee that the prime $p$ is unramified in 
$K=\QQ(\sqrt{-3}), \QQ(\sqrt{-11}), \QQ(\sqrt{-19}), \QQ(\sqrt{-43})$. By the discussion in Chapter \ref{section:irreducibility}, we know that the determinant of $\bar{\rho}_{E, p}$ is the mod $p$ cyclotomic character and $\bar{\rho}_{E, p}$ is finite flat at every prime $\mathfrak{p}$ of $K$ that lies above $p$. Therefore, Serre's conjecture is applicable in our setting, and we proceed as in \cite{cturcacs2018fermat}: 

Let $K$ be a number field with the ring of integers $\mcalO_K$ and let $\mathcal{N}$ be an ideal of $K$. We will start with some definitions. Note that for the precise definition of the locally symmetric space $Y_0(\mathcal{N})$, the reader may refer to \cite[Section 2]{csengun2018asymptotic}.

We note, as explained in \cite{cturcacs2018fermat}, that explicit computations in the cohomology groups of locally symmetric spaces are considerably simpler when the field has class number one. This is one of the main reasons for restricting our attention to imaginary quadratic fields of class number one.

\begin{definition}[{\cite[Section 2.1]{kara2020asymptotic}}] A weight $2$ complex eigenform $f$ over $K$ of degree $i$ and level $\mathcal{N}$ is a ring homomorphism $f: \mathbb{T}_{\mathbb{C}}^{(i)}(\mathcal{N})\to \mathbb{C}$, where $\mathbb{T}_{\mathbb{C}}^{(i)}(\mathcal{N})$ is the commutative $\ZZ$-algebra inside the endomorphism algebra of $\operatorname{H}^i(Y_0(\mathcal{N}), \mathbb{C})$ generated by the Hecke operators $T_{\mathfrak{q}}$ for every prime $\mathfrak{q}\in Spec(\mcalO_K)$ not dividing $\mathcal{N}$.
\end{definition}

\begin{definition}[{\cite[Section 2.1]{kara2020asymptotic}}]
Let $p$ be a rational unramified prime in $K$ that is coprime to $\mathcal{N}$. A weight $2$ mod $p$ eigenform $\theta$ over $K$ of degree $i$ and level $\mathcal{N}$ is a ring homomorphism $\theta: \mathbb{T}_{\overline{\mathbb{F}}_p}^{(i)}(\mathcal{N})\to \overline{\mathbb{F}}_p$. Here, we only consider the Hecke operators $T_{\mathfrak{q}}$ where $\mathfrak{q}$ is a prime that is coprime to $p\mathcal{N}$.
\end{definition}

\begin{definition}[{\cite[Section 2.1]{kara2020asymptotic}}]
For a weight $2$ complex eigenform $f$ over $K$ of degree $i$ and level $\mathcal{N}$, denote the number field generated by the values of $f$ by $\mcalO_f$. Let $\theta$ be a weight $2$ mod $p$ eigenform over $K$ of degree $i$ and level $\mathcal{N}$. We say that $\theta$ lifts to a complex eigenform if there is a weight $2$ complex eigenform $f$ over $K$ of degree $i$ and level $\mathcal{N}$, and a prime ideal $\mathfrak{p}$ of $\mcalO_f$ lying above $p$ such that $$\theta(T_{\mathfrak{q}})\equiv f(T_{\mathfrak{q}}) \quad(\bmod\text{ } \mathfrak{p})$$for every prime $\mathfrak{q}$ of $K$ that is coprime to $p\mathcal{N}$. 
\end{definition}

\subsection{Abelianization}
Next, we proceed as in \cite[Section 5]{isik2023ternary}. Let $p$ be a rational prime such that $p$ is unramified in $K$ and relatively prime to $\mathcal{N}$. Consider the following short exact sequence, which is given by the multiplication-by-$p$ map:
$$
0 \longrightarrow \mathbb{Z} \xrightarrow{\times p} \mathbb{Z} \longrightarrow \mathbb{F}_p \longrightarrow 0 \text {. }
$$Set $Y=Y_0(\mathcal{N})$. From the long exact sequence of cohomology groups associated to the above short exact sequence, we have the following diagram:

\begin{footnotesize}
$$
\begin{tikzcd}[row sep=0.5em,column sep=0.5em]
&& \ldots \arrow[rr]& & \operatorname{H}^1(Y,\ZZ) \arrow[rr]  & & \operatorname{H}^1(Y,\ZZ) \arrow[rr, "\times p"]  & & \operatorname{H}^1(Y,\mathbb{F}_p) \arrow[rr, "\delta"] \arrow[red, dr]
 & & \operatorname{H}^2(Y,\ZZ) \arrow[, rr]  & & \operatorname{H}^2(Y,\ZZ) \arrow[rr]  & & \ldots  \\
&& &  & & & & \coker(\times p)=\Ker(\delta) \arrow[red, ur] & &  \im(\delta)=\Ker(\times p) \arrow[red, dr]  & & & &  \\
&&  & & & & 0 \arrow[red, ur] &&   && 0 & & & & & & &&&
\end{tikzcd}
$$
\end{footnotesize}

This diagram gives rise to the following short exact sequence:

$$
0 \longrightarrow \Ker(\delta) \longrightarrow \operatorname{H}^1(Y,\mathbb{F}_p) \longrightarrow \im(\delta) \longrightarrow 0 \text {. }
$$

Let us compute $\Ker(\delta)$ and $\im(\delta)$:
$$\Ker(\delta)=\coker(\times p)=\operatorname{H}^1(Y,\ZZ)/p\operatorname{H}^1(Y,\ZZ)\cong \operatorname{H}^1(Y,\ZZ)\otimes\mathbb{F}_p$$and
$$\im(\delta)=\Ker(\times p)=\Ker\left(\operatorname{H}^2(Y,\ZZ)\to \operatorname{H}^2(Y,\ZZ)\right)=\operatorname{H}^2(Y,\ZZ)[p]$$Therefore, we derived that the following sequence is exact:

$$
0 \longrightarrow \operatorname{H}^1(Y,\ZZ)\otimes\mathbb{F}_p \longrightarrow \operatorname{H}^1(Y,\mathbb{F}_p) \longrightarrow \operatorname{H}^2(Y,\ZZ)[p] \longrightarrow 0 \text {. }
$$Clearly, $\operatorname{H}^2(Y,\ZZ)[p]=0$ if and only if the map $\operatorname{H}^1(Y,\ZZ)\otimes\mathbb{F}_p \longrightarrow \operatorname{H}^1(Y,\mathbb{F}_p)$ is surjective and this holds if and only if $\operatorname{H}^1(Y,\ZZ)\to \operatorname{H}^1(Y,\mathbb{F}_p)$ is surjective. If this is the case, then we see that any Hecke eigenvector $\overline{\alpha}$ in $\operatorname{H}^1(Y,\mathbb{F}_p)$ comes from such an eigenvector in $\operatorname{H}^1(Y,\ZZ)\otimes\mathbb{F}_p $.

On the other hand, considering the $p$-adic integers $\ZZ_p$ yields the following short exact sequence:$$
    0 \longrightarrow \mathbb{Z}_p \xrightarrow{\times p} \mathbb{Z}_p \longrightarrow \mathbb{F}_p \longrightarrow 0 \text {. }
$$This short exact sequence gives rise to an analogous short exact sequence on cohomology groups: 
$$
0 \longrightarrow \operatorname{H}^1(Y,\ZZ_p)\otimes\mathbb{F}_p \longrightarrow \operatorname{H}^1(Y,\mathbb{F}_p) \longrightarrow \operatorname{H}^2(Y,\ZZ_p)[p] \longrightarrow 0 \text {. }
$$As explained previously, if $\operatorname{H}^2(Y,\ZZ_p)[p]$ has only trivial $p$-torsion, then any Hecke eigenvector $\overline{\alpha}$ in $\operatorname{H}^1(Y,\mathbb{F}_p)$ comes from such an eigenvector in $\operatorname{H}^1(Y,\ZZ_p)\otimes\mathbb{F}_p $. Now, a lifting lemma of Ash and Stevens is applicable: By \cite[Proposition 1.2.2]{ash1986cohomology}, we know that there exists a finite integral extension $R^{\prime}$ of $\ZZ_p$, a prime $\mathfrak{p}$ above $p$ and a Hecke eigenvector $\alpha \in \operatorname{H}^1(Y, R^{\prime})$ such that the Hecke eigenvalues of $\alpha$ modulo $\mathfrak{p}$ are the ones of $\overline{\alpha}$. By fixing an embedding $\QQ\to \mathbb{C}$, we can see $\alpha\in \operatorname{H}^1(Y, \mathbb{C})$.

The existence of a (complex or mod $p$) eigenform is equivalent to the existence of a class in the corresponding cohomology group that is a simultaneous eigenvector for the Hecke operators such
that its eigenvalues match the values of the eigenform. From this point of view, we deduce that a mod $p$ eigenform lifts to a complex one whenever $\operatorname{H}^2(Y,\ZZ_p)$ has no non-trivial $p$-torsion.

Note that $\operatorname{H}^2(Y,\ZZ)\otimes \ZZ_p \cong \operatorname{H}^2(Y, \ZZ_p)$. This comes from the fact that  $\ZZ_p$ is flat over $\ZZ$, the universal coefficient theorem for cohomology and properties of the tensor product. Therefore, the cohomology groups $\operatorname{H}^2(Y,\ZZ)$ and $\operatorname{H}^2(Y,\ZZ_p)$ have the same $p$-torsion structure.

Recall from the beginning of Section \ref{section:irreducibility} that the Serre conductor $\mathcal{N}$ of the mod $p$ representation $\bar{\rho}_{E, p}$ is supported only on $S=\{\mathfrak{p}:\mathfrak{p}\mid2d\}$. For each imaginary quadratic number field $K$, the form of $\mathcal{N}$ can be found in Table \ref{conductortable}.

From Serre's modularity conjecture, we obtain a mod $p$ eigenform $\Psi: \mathbb{T}_{\mathbb{F}_p}\left(Y_0(\mathcal{N})\right) \rightarrow \overline{\mathbb{F}}_p$ over $K$ such that for every prime $\mathfrak{q}$ of $\mcalO_K$ that is coprime to $p\mathcal{N}$, $$
\operatorname{Tr}\left(\bar{\rho}\left(\operatorname{Frob}_{\mathfrak{q}}\right)\right)=\Psi\left(T_{\mathfrak{q}}\right) .
$$
Note that the image $\operatorname{Tr}\left(\bar{\rho}\left(\operatorname{Frob}_{\mathfrak{q}}\right)\right)$ is in $\mathbb{F}_p$, which implies that $\Psi$ corresponds to a class in $\operatorname{H}^1\left(Y_0(\mathcal{N}), \mathbb{F}_p\right)$ that is an eigenvector for all such Hecke operators $T_{\mathfrak{q}}$. 

Now, let us explain the relation between $\operatorname{H}^2(Y,\ZZ)[p]$ and $\Gamma_0(\mathcal{N})^{ab}$: Note that

\begin{eqnarray*}
    \operatorname{H}^2(Y,\ZZ)\otimes \mathbb{Z}\left[\frac{1}{6}\right] \cong \operatorname{H}^2\left(Y, \mathbb{Z}\left[\frac{1}{6}\right]\right) &\cong& \operatorname{H}^2\left(\Gamma_0(\mathcal{N}),\mathbb{Z}\left[\frac{1}{6}\right]\right)\\
    &\cong& H_1\left(\Gamma_0(\mathcal{N}),\mathbb{Z}\left[\frac{1}{6}\right]\right)
\end{eqnarray*}These isomorphisms show that if $\operatorname{H}^2(Y,\ZZ)$ has a $p$-torsion, then $\operatorname{H}^2\left(Y, \mathbb{Z}\left[\frac{1}{6}\right]\right)$ has a $p$-torsion. Therefore, $H_1\left(\Gamma_0(\mathcal{N}),\mathbb{Z}\left[\frac{1}{6}\right]\right)$ has a $p$-torsion. Since $$\Gamma_0(\mathcal{N})^{ab}[p]\otimes\mathbb{Z}\left[\frac{1}{6}\right]\cong H_1(\Gamma_0(\mathcal{N}),\ZZ)[p]\otimes\mathbb{Z}\left[\frac{1}{6}\right]\cong H_1\left(\Gamma_0(\mathcal{N}), \mathbb{Z}\left[\frac{1}{6}\right]\right)[p],$$ it follows that $\Gamma_0(\mathcal{N})$ has a non-trivial $p$-torsion.

After slightly modifying the code used by Turcas in \cite{cturcacs2018fermat}, \cite{cturcacs2020serre}, we computed the abelianization $\Gamma_0(\mathcal{N})^{a b}$ and its torsion for the levels $\mathcal{N}$ as in Table \ref{conductortable}. We provide them in Table \ref{primetortable} and the relevant codes can be found at 
\begin{center}
\url{https://warwick.ac.uk/fac/sci/maths/people/staff/turcas/fermatprog}.
\end{center}The scripts and their outputs are available in the GitHub repository \url{https://github.com/BGCakti/Solving-Fermat-type-equations-over-quadratic-fields}.

\subsection{Proof of Theorem \ref{mainthm2:d=-3} and Theorem \ref{mainthm3}}\label{pfthm3}

We are now ready to finalize the proof of Theorem \ref{mainthm2:d=-3} and Theorem \ref{mainthm3}. Recall that for $K=\QQ(\sqrt{-3})$, we assumed $p>13$ whereas for $K=\QQ(\sqrt{-11}),\QQ(\sqrt{-19}),\QQ(\sqrt{-43})$, we assumed $p>\max\{B_K,l_K,p_K\}$ where the constants are defined as in Theorem \ref{mainthm3}. By the fact that $\Gamma_0(\mathcal{N})^{a b}[p]=0$ for the indicated bounds on $p$, we see that the mod $p$ eigenforms lift to complex ones in each case, i.e. there is a complex eigenform $\mathfrak{f}$ over $K$ of level $\mathcal{N}$ as indicated in Table \ref{conductortable} such that for all prime ideals $\mathfrak{q}$ coprime to $p \mathcal{N}$ we have
\begin{equation}\label{complexeig}
\operatorname{Tr}\left(\bar{\rho}_{E, p}\left(\operatorname{Frob}_{\mathfrak{q}}\right)\right) \equiv \mathfrak{f}\left(T_{\mathfrak{q}}\right) \pmod{\mathfrak{p}},   
\end{equation}where $\mathfrak{p}$ is a prime ideal of $\mathbb{Q}_{\mathfrak{f}}$ lying above $p$ and $\mathbb{Q}_{\mathfrak{f}}$ is the number field generated by its eigenvalues. Throughout the rest of the paper, we denote the relation (\ref{complexeig}) by $\bar{\rho}_{E, p} \sim \bar{\rho}_{\mathfrak{f}, \mathfrak{p}}
$.

Recall that in Section \ref{techniques} we mentioned two methods to eliminate Hilbert newforms arising at a certain level by bounding $p$ above, namely Lemma \ref{C_f} and Lemma \ref{lemma:imageofinertia}. The latter holds for general number fields. However, the former holds for totally real fields so we have to make sure that an analogue of Lemma \ref{C_f} is true for number fields that have complex embeddings. The following lemma compensates for that. 

\begin{lemma}\label{boundingp}
Let $\mathfrak{q}\nmid p\mathcal{N}$ be a prime of $K$. Let $\mathfrak{f}$ be a Bianchi newform of level dividing $\mathcal{N}$. Define
$$
\mathcal{A}_{\mathfrak{q}}=\{a \in \mathbb{Z}:\text{ }|a| \leq 2 \sqrt{\operatorname{Norm}(\mathfrak{q})},\text{ } \operatorname{Norm}(\mathfrak{q})+1-a \equiv 0 \pmod 4\}.
$$If $\bar{\rho}_{E, p} \sim \bar{\rho}_{\mathfrak{f}, \mathfrak{p}}
$, then the prime $\mathfrak{p}$ divides 
$$
B_{\mathfrak{f}, \mathfrak{q}}:=\operatorname{Norm}(\mathfrak{q}) \cdot\left(\operatorname{Norm}(\mathfrak{q}+1)^2-a_{\mathfrak{q}}(\mathfrak{f})^2\right) \cdot \prod_{a \in \mathcal{A}_{\mathfrak{q}}}\left(a-a_{\mathfrak{q}}(\mathfrak{f})\right) \mathcal{O}_{\mathbb{Q}_{\mathfrak{f}}}
$$
\end{lemma}

\begin{proof}
Let $\mathfrak{q}\nmid p\mathcal{N}$ be a prime of $K$. Let $\mathfrak{f}$ be a Bianchi newform of level dividing $\mathcal{N}$ such that $\bar{\rho}_{E, p} \sim \bar{\rho}_{\mathfrak{f}, \mathfrak{p}}
$. If $\mathfrak{q}\mid p$, then by taking norms we see that $p\mid \Norm(\mathfrak{q})$. This implies $\mathfrak{p}\mid B_{\mathfrak{f}, \mathfrak{q}}$. Now, assume that $\mathfrak{q}\nmid p$. By Lemma \ref{Esem}, $E$ is semistable at $\mathfrak{q}$. If $E$ has good reduction at $\mathfrak{q}$, together with the relation $\bar{\rho}_{E, p} \sim \bar{\rho}_{\mathfrak{f}, \mathfrak{p}}
$, we have
$$\operatorname{Tr}\left(\bar{\rho}_{E, p}\left(\operatorname{Frob}_{\mathfrak{q}}\right)\right)\equiv a_{\mathfrak{q}}(E) \equiv a_{\mathfrak{q}}(\mathfrak{f}) \pmod {\mathfrak{p}},$$where $a_{\mathfrak{q}}(E)=\Norm(\mathfrak{q})+1-\#E\left(\mcalO_K/\mathfrak{q}\right)$. Since $\mathfrak{q}\nmid 2$, the map $E(K)[2]\to E\left(\mcalO_K/\mathfrak{q}\right)$ is injective. This shows that $4\mid \#E\left(\mcalO_K/\mathfrak{q}\right)=\Norm(\mathfrak{q})+1-a_{\mathfrak{q}}(E)$ as $E$ has full $2$-torsion over $K$. On the other hand, Hasse-Weil bound implies $|a_{\mathfrak{q}}(E)|\le 2\sqrt{\Norm(\mathfrak{q})}$. This shows that $a_{\mathfrak{q}}(E)\in \mathcal{A}_{\mathfrak{q}}$. If $E$ has multiplicative reduction at $\mathfrak{q}$, this time the relation $\bar{\rho}_{E, p} \sim \bar{\rho}_{\mathfrak{f}, \mathfrak{p}}
$ yields $$\pm (\Norm(\mathfrak{q})+1)\equiv a_{\mathfrak{q}}(\mathfrak{f}) \pmod {\mathfrak{p}}.$$Therefore, $\mathfrak{p}\mid \left(\operatorname{Norm}(\mathfrak{q}+1)^2-a_{\mathfrak{q}}(\mathfrak{f})^2\right)$ and $\mathfrak{p}\mid B_{\mathfrak{f}, \mathfrak{q}}$.
\end{proof}

Let us describe how Lemma \ref{boundingp} is useful: Let $B_{\mathfrak{f}}:=\sum_{\mathfrak{q} \in T} B_{\mathfrak{f}, \mathfrak{q}}$, where $T$ is the set of primes $\mathfrak{q}$ of $K$ such that $\mathfrak{q}\nmid p\mathcal{N}$ having norm less than $50$. Let $C_{\mathfrak{f}}=\Norm_{\QQ_\mathfrak{f}/\QQ}(B_{\mathfrak{f}})$. If $\bar{\rho}_{E, p} \sim \bar{\rho}_{\mathfrak{f}, \mathfrak{p}}$, Lemma \ref{boundingp} implies $\mathfrak{p}\mid B_{\mathfrak{f}}$ so that $p\mid C_{\mathfrak{f}}$. Therefore, if $p\nmid C_{\mathfrak{f}}$, then the relation $\bar{\rho}_{E, p} \sim \bar{\rho}_{\mathfrak{f}, \mathfrak{p}}$ cannot hold.

Using Magma, we computed the cuspidal Bianchi newforms $\mathfrak{f}$ at levels $\mfp^i\mathfrak{D}$ with $i\in\{0,1,2,3,4\}$, the Hecke eigenvalue fields $\QQ_{\mathfrak{f}}$ and eigenvalues $a_{\mathfrak{q}}({\mathfrak{f}})$ at primes $\mathfrak{q}$ of $K$ having small norm.

We will start with the cases where $d=-11,-19,-43$ and treat the case where $d=-3$ lastly, since the latter case is slightly more technical than the other three.

\begin{itemize}
    \item When $K=\QQ(\sqrt{-11})$, there is a unique Bianchi modular form at level $\mathfrak{D}$. For this modular form, $C_{\mathfrak{f}}$ is divisible by $3$ and $5$. There is no Bianchi modular form at level $\mfp\mathfrak{D}$. There is one Bianchi modular form at level $\mfp^2\mathfrak{D}$ for which $C_{\mathfrak{f}}$ is divisible by $3$ and $5$. There are $6$ Bianchi modular forms at level $\mfp^2\mathfrak{D}$. For these forms, $C_{\mathfrak{f}}$ is divisible by $2,3$ or $5$. There are $8$ Bianchi modular forms at level $\mfp^4\mathfrak{D}$. For these forms, $C_{\mathfrak{f}}$ is divisible by $2,3$ or $5$.
\end{itemize}

When $d=-19,-43$, our Magma routine could handle computations only up to the level $\mfp^3\mathfrak{D}$, $\mfp^2\mathfrak{D}$, respectively. Furthermore, for Bianchi modular forms, only the rational ones are listed in the LMFDB. Therefore, in an earlier version of this paper, we were able to access only the rational Bianchi modular forms at the levels $\mfp^4\mathfrak{D}$ when $d=-19$ and $\mfp^i\mathfrak{D}$ with $i=3,4$ when $d=-43$. Regarding Theorem \ref{mainthm3}, this step had made our results effective and not explicit for $d=-19,-43$. 

Thanks to John Cremona, we have now overcome these limitations. Using his algorithm \cite{Cremona_bianchi_progs} for Bianchi modular forms, he kindly provided us with, for each level mentioned above, the splitting of the newspace, the irreducible factors of the characteristic polynomials $T_\mathfrak{p}$ for a suitable prime $\mathfrak{p}$ (i.e. a prime ideal $\mathfrak{p}$ at which the Hecke polynomial $T_\mathfrak{p}$ splits into linear factors), and hence the minimal polynomial of the eigenvalue of each irrational Bianchi modular form at $\mathfrak{p}$. This allowed us to detect irrational Bianchi modular forms and apply our algorithm to each individually to compute the norms $C_\mathfrak{f}$ associated with them. 

We may now start working on the cases $d=-19,-43$. 

\begin{table}
\centering
\begin{tabular}{|m{1cm} |m{2.7cm}| m{11cm}|}
\hline
$d$ & \text{Level} & \text{Prime factors of }$C_{\mathfrak{f}}$   \\
\hline
$-11$ & $\mfp^i\mathfrak{D}, i=1,2,3,4$  &  2, 3, 5\\
\hline
$-19$ & $\mfp^4\mathfrak{D}$  &  2, 3, 5, 11, 19, 23, 29, 31, 41, 47, 59, 61, 71, 79, 89, 101, 109, 131, 139, 149, 151, 179, 181, 191, 199, 211, 239, 251, 271, 281, 283, 311, 3359, 349, 3499, 409, 443, 463, 467, 479, 499, 521, 523, 541, 571, 661, 719, 739, 769, 839, 1019, 1151, 1201, 1291, 1439, 1583, 1669, 1759, 1889, 1931, 2039, 2111, 2659, 2699, 3347, 3359, 3499, 3919, 4051, 6199, 6551, 14879, 21589, 190649 \\
\hline
\multirow{2}{*}{$-43$} & $\mfp^3\mathfrak{D}$ &  2, 3, 5, 7, 11, 13, 17, 23, 29, 31, 37, 41, 53, 71, 127, 131, 541, 3347, 137443 \\ 
\cline{2-3}

                     & $\mfp^4\mathfrak{D}$ & 2, 3, 5, 7, 11, 13, 17, 19, 23, 29, 31, 37, 41, 43, 47, 53, 59, 61, 67, 71, 73, 79, 83, 89, 97, 103, 107, 113, 127, 131, 139, 151, 157, 163, 167, 181, 191, 199, 211, 223, 229, 239, 251, 271, 281, 311, 383, 419, 431, 433, 449, 479, 571, 599, 617, 739, 769, 811, 829, 947, 1019, 1021, 10289, 12697, 1327, 1931, 1999, 2141, 2687, 4217, 4673, 58543, 1427707019  \\
\hline
\end{tabular}
\caption{Prime factors of $C_{\mathfrak{f}}$ for the irrational forms $\mathfrak{f}$}
\label{table:cremona_output}
\end{table}

\begin{itemize}
    \item When $K=\QQ(\sqrt{-19})$, there is one Bianchi modular form at level $\mathfrak{D}$ and $2$ Bianchi modular forms at level $\mfp\mathfrak{D}$. For these modular forms, $C_{\mathfrak{f}}$ is divisible by $3$ and $5$. There are $3$ Bianchi modular forms at level $\mfp^2\mathfrak{D}$. For these forms, $C_{\mathfrak{f}}$ is divisible by $3,5$ or $7$. There are $4$ modular forms at level $\mfp^3\mathfrak{D}$. For each of these forms, either $C_{\mathfrak{f}}$ is divisible by $2,3$ or $7$ or we have $B_{\mathfrak{f}}=\mcalO_K$. For the level $\mfp^4\mathfrak{D}$, the total dimension of the new subspace is $23$. There are $6$ rational Bianchi modular forms and $8$ irrational Bianchi modular forms at this level. For each of the irrational Bianchi modular form $\mathfrak{f}$, the norm $C_{\mathfrak{f}}$ is divisible by combinations of primes appearing in Table \ref{table:cremona_output}. Since $p>C_K$, we eliminated all irrational Bianchi modular forms at this level. Now, we deal with the rational ones. For the sake of the reader, we will eliminate them manually, using the Hasse-Weil bound. However, one can also eliminate them by using our Magma routine and the data given in the LMFDB for each rational form. Those $6$ rational forms at level $\mfp^4\mathfrak{D}$ have LMFDB labels $2.0.19.1-4864.1-a,b,c,d,e,f$. Let $\mathfrak{f}$ be any of these Bianchi modular forms. Let $\mathfrak{q}\mid 5$. Using LMFDB, we see that $a_\mathfrak{q}(\mathfrak{f})=-4,-1,0,3$. Note that by Lemma \ref{Esem}, $E$ is semistable away from $\mfp$. In particular, $E$ is semistable at $\mathfrak{q}$. If $E$ has good reduction at $\mathfrak{q}$, then by \cite[Theorem V.2.3.1(b)]{silverman2009arithmetic}, and relation (\ref{complexeig}), we have:
$$
\operatorname{Tr}\left(\bar{\rho}_{E, p}\left(\operatorname{Frob}_{\mathfrak{q}}\right)\right) \equiv a_{\mathfrak{q}}(E)\equiv a_{\mathfrak{q}}(\mathfrak{f}) \quad(\bmod p);
$$where $a_{\mathfrak{q}}(E)=\Norm(\mathfrak{q})+1-\#E\left(\mcalO_K/\mathfrak{q}\right)$. Now, let us determine $\#E\left(\mcalO_K/\mathfrak{q}\right)$: Since $\mathfrak{q}\nmid 2$, the map $E(K)[2]\to E\left(\mcalO_K/\mathfrak{q}\right)$ is injective. This shows that $4\mid \#E\left(\mcalO_K/\mathfrak{q}\right)$ as $E$ has full $2$-torsion over $K$. On the other hand, Hasse-Weil bound $|a_{\mathfrak{q}}(E)|\le 2\sqrt{\Norm(\mathfrak{q})}=2\sqrt{5}$ forces $\#E\left(\mcalO_K/\mathfrak{q}\right)=4,8$. Therefore, $a_{\mathfrak{q}}(E)=\pm 2$, which implies that $p\mid \pm2 -a_{\mathfrak{q}}(\mathfrak{f})=-5,-1,\pm 2,3,6$. But this is a contradiction since $p>C_K$. If $E$ has (either split or non-split) multiplicative reduction at $\mathfrak{q}$, then the relation $(\ref{complexeig})$ yields
$$\pm (\Norm(\mathfrak{q})+1)\equiv a_{\mathfrak{q}}(\mathfrak{f}) \quad(\bmod p).$$Therefore, $p\mid \pm6-a_{\mathfrak{q}}(\mathfrak{f})$ i.e $p\mid -9,-5,\pm2,3,7,10$ which is a contradiction since $p>C_K$. 
    
    \item When $K=\QQ(\sqrt{-43})$, there are $3$ Bianchi modular forms at level $\mathfrak{D}$, for which $C_{\mathfrak{f}}$ is divisible by $3$ or $7$ or $B_{\mathfrak{f}}=\mcalO_K$. There are $2$ Bianchi modular forms at level $\mfp\mathfrak{D}$. For these modular forms, $C_{\mathfrak{f}}$ is divisible by $3,5,7$ or $11$. There are $4$ Bianchi modular forms at level $\mfp^2\mathfrak{D}$. For these modular forms, $C_{\mathfrak{f}}$ is divisible by $3$ or $5$ or $B_{\mathfrak{f}}=\mcalO_K$. We eliminate the Bianchi modular forms at level $\mfp^3\mathfrak{D}$ and $\mfp^4\mathfrak{D}$: Total dimension of the new cuspidal subspace at the level $\mfp^3\mathfrak{D}$, $\mfp^4\mathfrak{D}$ is $15$ and $39$, respectively. There is one rational Bianchi modular form and $5$ irrational Bianchi modular forms at level $\mfp^3\mathfrak{D}$. For the irrational ones, $C_{\mathfrak{f}}$ is divisible by at least one of the primes appearing in Table \ref{table:cremona_output}. The unique rational Bianchi modular form $\mathfrak{f}$ at level $\mfp^3\mathfrak{D}$ has LMFDB label $2.0.43.1-2752.1-a$. Let $\mathfrak{q}\mid 11$. Note that $a_\mathfrak{q}(\mathfrak{f})=1$. By Lemma \ref{Esem}, $E$ is semistable at $\mathfrak{q}$. If $E$ has good reduction at $\mathfrak{q}$, then we have:
$$
\operatorname{Tr}\left(\bar{\rho}_{E, p}\left(\operatorname{Frob}_{\mathfrak{q}}\right)\right) \equiv a_{\mathfrak{q}}(E)\equiv a_{\mathfrak{q}}(\mathfrak{f}) \pmod{p};
$$where $a_{\mathfrak{q}}(E)=\Norm(\mathfrak{q})+1-\#E\left(\mcalO_K/\mathfrak{q}\right)$. Now, let us determine $\#E\left(\mcalO_K/\mathfrak{q}\right)$: Since $\mathfrak{q}\nmid 2$, the map $E(K)[2]\to E\left(\mcalO_K/\mathfrak{q}\right)$ is injective. This shows that $4\mid \#E\left(\mcalO_K/\mathfrak{q}\right)$ as $E$ has full $2$-torsion over $K$. On the other hand, Hasse-Weil bound $|a_{\mathfrak{q}}(E)|\le 2\sqrt{\Norm(\mathfrak{q})}=2\sqrt{11}$ forces $\#E\left(\mcalO_K/\mathfrak{q}\right)=8,12,16$. Therefore, $a_{\mathfrak{q}}(E)=0,\pm4$, which implies that $p\mid -5,-1,3$. But this is a contradiction since $p>C_K$. If $E$ has (either split or non-split) multiplicative reduction at $\mathfrak{q}$, then the relation $(\ref{complexeig})$ yields
$$\pm (\Norm(\mathfrak{q})+1)\equiv a_{\mathfrak{q}}(\mathfrak{f}) \pmod{p}.$$Therefore, $p\mid \pm12-a_{\mathfrak{q}}(\mathfrak{f})=-13,11$. There are $3$ rational Bianchi modular forms and $16$ irrational forms at level $\mfp^4\mathfrak{D}$. For the irrational ones, $C_{\mathfrak{f}}$ is divisible by at least one of the primes appearing in Table \ref{table:cremona_output}. The rational ones correspond to Bianchi modular forms in the LMFDB label $2.0.43.1-11008.1-a,b,c$. We have $a_{\mathfrak{q}}(\mathfrak{f})=-1,\pm 3$ where $\mathfrak{q}\mid 11$. By comparing the trace of the image of Frobenius at $\mathfrak{q}$ as we did previously, we see that $p=1,3,7$ if $E$ has good reduction at $\mathfrak{q}$ and $p=3,5,11,13$ if $E$ has multiplicative reduction at $\mathfrak{q}$. In all cases, we have $p>C_K$, and this cannot happen.

\end{itemize}

All Bianchi modular forms at level $\mathcal{N}$ are eliminated when $K=\QQ(\sqrt{-11}),\QQ(\sqrt{-19}),\QQ(\sqrt{-43})$. Hence, Theorem \ref{mainthm3} follows for the previous fields $K$. Let us now discuss the remaining case where $K=\QQ(\sqrt{-3})$. Recall that we assumed an additional conjecture in this case, namely Conjecture \ref{fakellcurve}. We recall also that the Bianchi modular forms $\mathfrak{f}$ satisfying $C_\mathfrak{f}=0$ require extra work.

\begin{itemize}
    \item When $K=\QQ(\sqrt{-3})$, there are no Bianchi modular forms at level $\mfp^i\mathfrak{D}$ for $i=0,1,2$. There is one Bianchi modular form at level $\mfp^3\mathfrak{D}$ and $\mfp^4\mathfrak{D}$. Let us call these forms $\mathfrak{f}_3$ and $\mathfrak{f}_4$. Both of them satisfy $C_{\mathfrak{f}_3}=C_{\mathfrak{f}_4}=0$ and $\QQ_{\mathfrak{f}_3}=\QQ_{\mathfrak{f}_4}=\QQ$. The LMFDB label of the Bianchi modular form $\mathfrak{f}_3$,$\mathfrak{f}_4$ is $2.0.3.1-192.1-a$, $2.0.3.1-768.1-a$, respectively. Note that they are both non-trivial and new. Conjecture \ref{fakellcurve} then predicts that there is either an elliptic curve of conductor $\mfp^i\mathfrak{D}$ or a fake elliptic curve of conductor $(\mfp^i\mathfrak{D})^2$ associated to the form $\mathfrak{f}_i$ for $i=3,4$. The elliptic curves in the isogeny classes given in the LMFDB label $2.0.3.1-192.1-a,2.0.3.1-768.1-a$ correspond to the forms $\mathfrak{f}_3,\mathfrak{f}_4$, respectively. All the elliptic curves in the previous isogeny classes have potentially good reduction at $\mfp$. Since the elliptic curve $E$ has potentially multiplicative reduction at the prime $\mfp$, we have $p\mid 24$, which is a contradiction. To be able to finalize the proof of Theorem \ref{mainthm3}, we still need to deal with the possibility that the forms $\mathfrak{f}_3,\mathfrak{f}_4$ correspond to a fake elliptic curve. We will closely follow the ideas from \cite{csengun2018asymptotic} to show that this cannot happen. We recall once again that the Frey curve $E$ has potentially multiplicative reduction at $\mfp$ and $p\nmid v_{\mfp}(j_E)$.
    \begin{lemma}
        If $p> 13$, the forms $\mathfrak{f}_3,\mathfrak{f}_4$ cannot correspond to fake elliptic curves.
    \end{lemma}
    \begin{proof}
        Let $\mathfrak{f}$ be one of the forms $\mathfrak{f}_3,\mathfrak{f}_4$. Assume not and say $\mathfrak{f}$ corresponds to a fake elliptic curve $A_\mathfrak{f}$. Note that $p \mid \# \bar{\rho}_{E, p}\left(I_{\mfp}\right)$ by Lemma \ref{divinertia}. Since $\mathfrak{f}$ corresponds to the fake elliptic curve $A_\mathfrak{f}$, Theorem \ref{fakethm} below yields $\#\bar{\rho}_{A_\mathfrak{f}, p}\left(I_{\mfp}\right)\le 24$. Since $\bar{\rho}_{E, p}\sim \bar{\rho}_{A_\mathfrak{f}, p}$ and $p> 13$, we have a contradiction. 
    \end{proof}
    \begin{theorem}[{\cite[Theorem 4.2]{csengun2018asymptotic}}]\label{fakethm}
Let $A / K$ be a fake elliptic curve. Then $A$ has potential good reduction everywhere. More precisely, let $\mathfrak{q}$ be a prime of $K$ and consider $A / K_{\mathfrak{q}}$. There is totally ramified extension $K^{\prime} / K_{\mathfrak{q}}$ of degree dividing 24 such that $A / K^{\prime}$ has good reduction. 
\end{theorem}
    
\end{itemize}

\begin{table}
$$\begin{array}{|c|c|c|c|}
\hline K & \text{Conductor  $\mathcal{N}$}  &\text{Condition on $i$} &\text{Ideal $\mathfrak{D}$} \\
\hline \QQ(\sqrt{-3}) & \mfp^i\mathfrak{D}  & i\in\{0,1,2,3,4\} &\mathfrak{D}\mid 3\\
\hline \QQ(\sqrt{-11}) & \mfp^i\mathfrak{D} & i\in\{0,1,2,3,4\} &\mathfrak{D}\mid 11\\
\hline \QQ(\sqrt{-19}) & \mfp^i\mathfrak{D}  & i\in\{0,1,2,3,4\}&\mathfrak{D}\mid 19\\
\hline \QQ(\sqrt{-43}) & \mfp^i\mathfrak{D}  & i\in\{0,1,2,3,4\} &\mathfrak{D}\mid 43\\
\hline
\end{array}$$
\caption{Serre conductor for imaginary quadratic fields $K$}
\label{conductortable}
\end{table}

\begin{table}
\begin{tabular}{| m{1.5cm} |m{13cm}| m{1cm} |}
\hline K & \text{Primes $l$ such that  $\Gamma_0(\mathcal{N})^{a b}[l] \neq 0$}\\ 
\hline $\QQ(\sqrt{-3})$ & 2, 3 \\ 
\hline $\QQ(\sqrt{-11})$ & 2, 3 \\
\hline $\QQ(\sqrt{-19})$ & 2, 3, 5, 17, 173, 199  \\
\hline $\QQ(\sqrt{-43})$ & 2, 3, 5, 7, 11, 13, 17, 19, 29, 37, 43, 47, 59, 61, 79, 107, 127, 139, 277,
307, 389, 613, 1031, 1523, 1823, 3019, 4657, 12659, 74821, 77681, 87583, 237581,
944659, 2180587, 3854969, 34315907 \\ 
\hline
\end{tabular}
\caption{Prime torsion in $\Gamma_0(\mathcal{N})^{a b}$}
\label{primetortable}
\end{table}

\section{Appendix: Limitations of existing methods in eliminating newforms}

We finalize our paper by discussing the limitations of the modular method, addressing the problems arising from real and imaginary quadratic fields separately.

\subsection{Real quadratic fields}

As we have already observed, the quadratic fields $K$ considered in Theorem \ref{mainthm1} are those in which there is a unique prime above $2$ and a unique odd ramified prime. However, there is no theoretical obstruction to applying the modular method to study the $d$-Fermat equation over other quadratic fields $K$ that do not satisfy either of these conditions. We provide examples explaining these two cases:

\begin{itemize}
    \item When $d=17$, let us study the $d$-Fermat equation over $K=\QQ(\sqrt{d})$. Note that $2\mcalO_K=\mfp_1\mfp_2$ and the only odd ramified prime in $K$ is $d$. Lemma \ref{semsq} implies that the Frey curve $E/K$ attached to a hypothetical solution $(a,b,c)$ is semistable at the prime above $d$. The proof of Lemma \ref{Esem} shows that $E$ is semistable away from $\mfp_1$ and $\mfp_2$. By Lemma \ref{potmultb}, we observe that $E$ has potentially multiplicative reduction at $\mfp_1$ and $\mfp_2$. By applying Lemma \ref{samir4.4}, it follows that the conductor $E$ is of the form $$N_E=\mfp_1\mfp_2^l\prod_{\substack{\mathfrak{D} \mid d }} \mathfrak{D} \prod_{\substack{\mathfrak{p} \mid a b c \\ \mathfrak{p} \nmid 2d}} \mathfrak{p};\text{ $l$=1,4.}$$The representation $\bar{\rho}_{E, p}$ is irreducible when $p>7$ by Lemmata \ref{irred11}, \ref{irred13}, \ref{irred17}, \ref{irred19}. Applying Theorem \ref{levellow} shows that the lowered level $\mathcal{N}$ is $$\mfp_1\mfp_2^i\mathfrak{D};\text{ where $i=1,4$ and $\mathfrak{D}\mid 17$},$$so we need to eliminate the Hilbert newforms at level $\mathcal{N}$. There is a unique newform $\mathfrak{f}$ at level $\mfp_1\mfp_2\mathfrak{D}$. As in the notation of Lemma \ref{C_f}, it satisfies $\QQ_\mathfrak{f}=\QQ$ and $C_\mathfrak{f}=0$. The elliptic curves in the isogeny class in LMFDB label $2.2.17.1-68.1-a$ correspond to $\mathfrak{f}$. There are $8$ elliptic curves in this class. Unfortunately, they all satisfy $v_{\mfp_i}(j)<0$ for $i=1,2$, where $j$ is the $j$-invariant of any elliptic curve in this isogeny class. Therefore, Lemma \ref{lemma:imageofinertia} is not applicable and we cannot eliminate the Hilbert newform $\mathfrak{f}$ at level $\mfp_1\mfp_2\mathfrak{D}$.
    \item When $d=21$, let us study the $d$-Fermat equation over $K=\QQ(\sqrt{d})$. Note that $2\mcalO_K=\mfp$ and there are two odd ramified primes in $K$, namely $3$ and $7$. In this case, applying level-lowering theorem implies that we need to eliminate Hilbert newforms at level $$\mathcal{N}=\mfp^l\mathfrak{D}_1\mathfrak{D}_2;\text{ where $l=1,4$ and $\mathfrak{D}_1^2\mathfrak{D}_2^2=21\mcalO_K$}.$$There are $4$ Hilbert newforms at level $\mfp\mathfrak{D}_1\mathfrak{D}_2$. Applying Lemma \ref{C_f}, we see that two of these forms satisfy $C_\mathfrak{f}\mid 2^85^2$. The remaining two satisfy $\QQ_\mathfrak{f}=\QQ$ and $C_\mathfrak{f}=0$ so we need to check if Lemma \ref{lemma:imageofinertia} is applicable. The elliptic curves in isogeny classes in LMFDB labels $2.2.21.1-84.1-a$ and $2.2.21.1-84.1-b$ correspond to these newforms. However, all the elliptic curves in these two isogeny classes have potentially multiplicative reduction at $\mfp$. Therefore, Lemma \ref{lemma:imageofinertia} is not applicable and we cannot eliminate the Hilbert newform $\mathfrak{f}$ at level $\mfp\mathfrak{D}_1\mathfrak{D}_2$. However, using the techniques presented in this paper, it is possible to prove the following theorem:
    \begin{theorem}\label{thm:l-fermat;21}
        Let $K=\QQ(\sqrt{21})$ and let $l$ be an odd ramified prime in $K$. The generalized Fermat equation 
        \begin{equation}\label{eqn:l-fermat}
         l^ra^p+b^p+c^p=0   
        \end{equation}has no non-trivial solutions $(a,b,c)\in K^3$ such that $(la,b,c)$ is primitive and $2\mid abc$ when $p>7$.
    \end{theorem}
    \item When $d=6,14$, one may prove the following result, analogous to Theorem \ref{thm:l-fermat;21}:
    \begin{theorem}
       Let $K=\QQ(\sqrt{6}),\QQ(\sqrt{14})$ and let $l$ be an odd ramified prime in $K$. Equation (\ref{eqn:l-fermat}) has no non-trivial solutions $(a,b,c)\in K^3$ such that $(la,b,c)$ is primitive when $p>5$ if $K=\QQ(\sqrt{6})$, and $p>11$ if $K=\QQ(\sqrt{14})$.
    \end{theorem}We give a sketch of the proof of the theorem: Applying the same steps as in Lemma \ref{semsq} and Lemma \ref{Esem} shows that the elliptic curve $$\text{$E_l$:      } y^2=x(x-l^r a^p)(x+b^p)$$is semistable away from the unique prime above $2$, $\mfp$. By Lemma \ref{potmultb}, $E_l$ has potentially multiplicative reduction at $\mfp$. Applying Lemma \ref{samir4.4} shows that the conductor of $E_l$ in these cases is given by $$N_{E_l}=\mfp^t\prod_{\substack{\mathfrak{L} \mid l }} \mathfrak{L} \prod_{\substack{\mathfrak{p} \mid a b c \\ \mathfrak{p} \nmid 2l}} \mathfrak{p};\text{ $t$=1,8.}$$By Lemmata \ref{irred7}, \ref{irred11}, \ref{irred13}, \ref{irred17}, we see that the Galois representation $\bar{\rho}_{E_l,p}$ is irreducible for $p=7,11,13,17$. For $p\ge 19$, we can show that $\bar{\rho}_{E_l,p}$ is irreducible by applying the same steps as in Lemma \ref{irred19} except for the fact that the the isogeny character $\theta$ is a character of the ray class group whic is either $\ZZ/2\ZZ$ or $\ZZ/2\ZZ\times\ZZ/4\ZZ$, i.e. $\theta$ might be a character of order $4$. However, we can discard this possibility: If the order of $\theta$ is $4$, let $L=\overline{K}^{\Ker(\theta^2)}$ be the unique quadratic extension of K cut out by $\theta^2$. (It is indeed a quadratic extension since $\operatorname{Gal}(L/K)=G_K/\Ker(\theta^2)\cong \{\pm 1\}$.) Then the restriction $\theta|_{G_L}$ is a quadratic character. Let $E^{\prime}$ be the quadratic twist of $E_l/L$ by $\theta|_{G_L}$. Then, $E^{\prime}$ is an elliptic curve defined over $L$, and 
    $$
\bar{\rho}_{E^{\prime}, p} \sim\theta|_{G_L}\cdot\bar{\rho}_{E, p}\sim\left(\begin{array}{cc}
\theta^2 & * \\
0 & \theta\theta^{\prime}
\end{array}\right)\sim\left(\begin{array}{cc}
1 & * \\
0 & \chi_p
\end{array}\right).
$$Therefore, $E^{\prime}$ has a point of order $p$ over $L$. The work \cite{derickx2023torsion} forces $p\le 17$. We reach a contradiction as we assumed $p\ge 19$. Now, applying Theorem \ref{levellow} yields that we need to eliminate Hilbert newforms at the levels $\mfp^t\mathfrak{L}$, where $t=1,8$ and $\mathfrak{L}\mid l$. When $d=6$, no Hilbert newforms at the level $\mfp\mathfrak{L}$ and there are $34$ Hilbert newforms at the level $\mfp^8\mathfrak{L}$. Among those $34$ newforms, $20$ of them -having LMFB labels $2.2.24.1-768.1-a,\ldots,-t$- are rational and they can be eliminated via Lemma \ref{lemma:imageofinertia}, namely the inertia argument. The remaining $14$ are irrational. Applying Lemma \ref{C_f} shows that $C_\mathfrak{f}$ is divisible by $2$ or $5$. When $d=14$, there are $4$ Hilbert newforms at the level $\mfp\mathfrak{L}$, and $200$ Hilbert newforms at the level $\mfp^8\mathfrak{L}$. Performing Magma computations to apply Lemma \ref{C_f} shows that the norm $C_\mathfrak{f}$ is divisible by $2,3,5$ or $11$ for each Hilbert newform $\mathfrak{f}$.

\item We also intended to study Equation (\ref{eqn:l-fermat}) over $\QQ(\sqrt{22})$ with $l=11$, for which we had to eliminate all Hilbert newforms at levels $\mfp^t\mathfrak{L}$ with $t=1,8$ and $\mathfrak{L}\mid 11$. However, we could not compute Hilbert newforms at the level $\mfp^8\mathfrak{L}$.
\end{itemize}

One may also try applying the modular method to study Fermat-type equations over other quadratic number fields. Although there are again no theoretical obstructions, computational challenges may arise if the number field under consideration has a large discriminant or if the norm of the lowered level is large. For example, we tried to study the $d$-Fermat equation over $K=\QQ(\sqrt{d})$ when $d=29$. There are $5$ Hilbert newforms at level $\mfp\mathfrak{D}$, and we could eliminate all of them using Lemma \ref{C_f}. There are $123$ Hilbert newforms at level $\mfp^4\mathfrak{D}$, at least $3$ of them satisfy $C_\mathfrak{f}=0$. We could not eliminate all $123$ forms at this level. In fact, our Magma implementation stopped responding while computing the norm $C_\mathfrak{f}$ for the last six forms. We remark that LMFDB currently lists all Hilbert newforms over $K=\QQ(\sqrt{29})$ up to level norm $995$, whereas the norm of the ideal $\mfp^4\mathfrak{D}$ is $7424$. Therefore, it was not possible to eliminate the remaining forms manually, as we did in some of the previous cases. 

\subsection{Imaginary quadratic fields}

Regarding imaginary quadratic fields, we attempted to study the $d$-Fermat equation over all imaginary quadratic fields $\QQ(\sqrt{d})$ of class number one in which $2$ is inert. Following the same steps until Section \ref{section:imaginaryquadratic} yields that the possible levels are $\mfp^i\mathfrak{D}$ with $i=0,1,2,3,4$ when $d=-67,-163$. Therefore, in each case, we have to make sure that the mod $p$ eigenforms arising from Serre's modularity conjecture lift to complex ones i.e., to compute the prime torsion in the abelianization $\Gamma_0(\mathcal{N})^{ab}$ and then we need to eliminate them using the lemmas stated in Section \ref{techniques}. When $d=-67$, unfortunately, we were not able to compute the prime torsion in $\Gamma_0(\mathcal{N})^{ab}$ when $\mathcal{N}=\mfp^4\mathfrak{D}$. The prime torsion $\Gamma_0(\mathcal{N})^{ab}$ for $\mathcal{N}=\mfp^i\mathfrak{D}$ with $i=0,1,2,3$ is given in Table \ref{d=-67:primetortable}. However, we were able to eliminate all  Bianchi modular forms at level $\mfp^i\mathfrak{D}$ with $i=0,1$, and the \textit{rational} ones at level $\mfp^i\mathfrak{D}$ with $i=2,3$. Indeed, there are $3$ Bianchi modular forms at level $\mathfrak{D}$, for which $C_{\mathfrak{f}}$ is divisible by $3,5$ or $11$. There are $2$ Bianchi modular forms at level $\mfp\mathfrak{D}$, for which $C_{\mathfrak{f}}$ is divisible by $3,5$ or $17$. There are no rational Bianchi modular forms at level $\mfp^3\mathfrak{D}$ and there is one rational Bianchi modular form at level $\mfp^2\mathfrak{D}$, having LMFDB label $2.0.67.1-1072.1-a$. The Hecke eigenvalue of this form at the prime $l\mid 17$ is $3$, and using the relation \ref{complexeig}, it is possible to eliminate it as we did when $d=-43$. We were able to eliminate all irrational Bianchi modular forms at level $\mfp^i\mathfrak{D}$ with $i=2,3$ thanks to the information provided by John Cremona. Prime factors of the norm $C_f$ for the irrational Bianchi modular forms at these are given in Table \ref{table:67_cremonaoutput}.

\begin{table}
\begin{tabular}{| m{1.5cm}|m{4.5cm}|m{8cm}|}
\hline Level & Number of irrational forms&\text{Primes $l$ such that  $\Gamma_0(\mathcal{N})^{a b}[l] \neq 0$}\\ 
\hline $\mfp^2\mathfrak{D}$ & 2 &3,5,7,11,17,29,37,61 \\ 
\hline $\mfp^3\mathfrak{D}$ & 5 &2, 3, 5, 7, 11, 13, 19, 23, 29, 37, 41, 53, 67, 73, 109, 193, 241, 479, 2777, 4513 \\

\hline
\end{tabular}
\caption{Remaining irrational forms over $\QQ(\sqrt{-67})$}
\label{table:67_cremonaoutput}
\end{table}

Therefore, the only obstacle to obtaining a result as in the cases $d=-19,-43$ is computational and concerns the level $\mathcal{N}=\mfp^4\mathfrak{D}$. By using \cite[Proposition 1]{papadopoulos1993neron}, we managed to avoid the level $\mfp^4\mathfrak{D}$ but this resulted in assuming (more) special conditions on the solutions of the $d$-Fermat equation. In particular, it is possible to derive the following result over $\QQ(\sqrt{-67})$. 

\begin{theorem}
Let $d=-67$ and let $K=\QQ(\sqrt{d})$. Assume that Conjecture \ref{serrconj} holds for $K$. Assume also that $\mfp^2\mid a_6+r a_4+r^2 a_2+r^3-t a_3-t^2-r t a_1$, where $r,t\in\mcalO_\mfp$ are as chosen in \cite[Proposition 1]{papadopoulos1993neron} and $a_1, a_2, a_3, a_4, a_6$ are the usual $a$-invariants of the elliptic curve having model (\ref{frey}).
Then, Equation (\ref{dfermat}) 

\begin{enumerate}[i)]
    \item has no non-trivial solutions $(a,b,c)$ over $K$ such that $(da,b,c)$ is primitive and $2\mid abc$ when $p>3323488887264568865776816914360851$.
    \item has no non-trivial solutions $(a,b,c)$ over $K$ such that $(da,b,c)$ is primitive when $p$ splits in $K$, $p\equiv 3\pmod4$ and $p>3323488887264568865776816914360851$.
\end{enumerate}
\end{theorem}

\begin{table}
\begin{tabular}{| m{1.3cm} |m{13cm}| m{1cm} |}
\hline K & \text{Primes $l$ such that  $\Gamma_0(\mathcal{N})^{a b}[l] \neq 0$}\\ 
\hline $\QQ(\sqrt{-67})$ & 2,3,5,7,11,13,19,29,37,41, 43, 59, 67, 101, 109, 163,  179, 359, 431, 673, 991,1481, 2521, 4793,5639, 23057, 32909, 48973, 138197,286381, 1833259,8827739, 22349783, 79626517, 104306779, 119593911,68867906617, 20888200728617,  763282887386969, 6221169664900085761,11132872059238913129,1212294248161911603264861232147, 3323488887264568865776816914360851 \\ 
\hline
\end{tabular}
\caption{Prime torsion in $\Gamma_0(\mathcal{N})^{a b}$}
\label{d=-67:primetortable}
\end{table}

When $d=-163$, although we were again able to eliminate the rational Bianchi modular forms at the corresponding levels, the algorithm for computing the prime torsion in $\Gamma_0(\mathcal{N})^{a b}$ was completely infeasible. Therefore, we could not derive an explicit result in this case concerning the solutions of $d$--Fermat. However, using the tools presented in this paper, one can see that the AGFC holds for the $d$--Fermat equation over $\QQ(\sqrt{-163})$ with an effective bound.

\bibliographystyle{plain}
\bibliography{references.bib}
\vspace{0.5cm}
\noindent
\textsc{Bogazici University, Istanbul, Turkiye; Galatasaray University, Istanbul, Turkiye; University of Groningen, Groningen, The Netherlands } \\
\emph{Email address: \texttt{begum.cakti@std.bogazici.edu.tr, bcakti@gsu.edu.tr, b.g.cakti@rug.nl}}
\end{document}